\documentclass[11pt,twoside,epsf]{article}
\usepackage{cite}
\usepackage{graphicx,subfigure}
\usepackage{latexsym,amsfonts,amsbsy,amssymb}
\usepackage{amsmath,amsthm}
\usepackage{amsmath,bm}
\usepackage{epstopdf}
\usepackage[colorlinks,linkcolor=cyan,anchorcolor=cyan,citecolor=cyan]{hyperref}

\newtheorem{method}{Method}[section]

\newtheorem{theorem}{Theorem}[section]

\newtheorem{lemma}{Lemma}[section]

\newcommand{\xu}{{\rm x}}
\newcommand{\yu}{{\rm y}}
\newcommand{\uu}{{\rm u}}

\def\Re{{\rm Re}}
\def\Im{{\rm Im}}

\makeatletter % '@' is now a normal "letter" for TeX
\@addtoreset{equation}{section}

\@addtoreset{figure}{section}

\@addtoreset{table}{section}
\makeatother

\begin{document}
\cleardoublepage
\pagestyle{myheadings}
\bibliographystyle{plain}

\title{On Relaxed Filtered Krylov Subspace Method for \\
Non-Symmetric Eigenvalue Problems}
\author{
Cun-Qiang Miao
\thanks{Supported by The National Natural Science Foundation of China (No. 11901361), P.R. China.}
\\[1mm]
{\it School of Mathematics and Statistics}\\
{\it Central South University}\\
{\it Changsha 410083, P.R. China}\\
{\it Email: cqmiao@csu.edu.cn} \\[3mm]
%%%%%%%%%%%%%%%%%%%%%%%%
Wen-Ting Wu
\thanks{Supported by Beijing Institute of Technology Research Fund Program for Young Scholars.}\\[1mm]
{\it School of Mathematics and Statistics} \\
{\it Beijing Institute of Technology} \\
{\it Beijing 100081, P.R. China} \\
{\it Email: wuwenting@bit.edu.cn}
}
\date{}
\markboth{\small C.-Q. Miao and W.-T. Wu}{\small On Relaxed Filtered Krylov Subspace Method}

%%%%%%%%%%%%%%%%%% maketitle %%%%%%%%%%%%%%%%%
\maketitle

%%%%%%%%%%%%%%% Begin Abstract %%%%%%%%%%%%%
\begin{abstract}
In this paper, by introducing a class of relaxed filtered Krylov subspaces, we propose the relaxed filtered Krylov subspace method for computing the eigenvalues with the largest real parts and the corresponding eigenvectors of non-symmetric matrices. As by-products, the generalizations of the filtered Krylov subspace method and the Chebyshev-Davidson method for solving non-symmetric eigenvalue problems are also presented. We give the convergence analysis of the complex Chebyshev polynomial, which plays a significant role in the polynomial acceleration technique. In addition, numerical experiments are carried out to show the robustness of the relaxed filtered Krylov subspace method and its great superiority over some state-of-the art iteration methods.
\end{abstract}

%%%%% Keywords %%%%%%%%%%%
{\small{\bf Keywords:} \quad
eigenvalue, relaxed filtered Krylov subspace, Chebyshev polynomial, non-symmetric matrix.}

%%%% AMS subject classifications %%%%
{\small{\bf AMS(MOS) Subject Classifications:}} \quad
65F15, 65N25.

%%%%%%%%%%%%%%%%%%% Start %%%%%%%%%%%%%%%%%%
\section{Introduction}
\label{secion-intr}
Denote by $\mathbb{R}^n$ and $\mathbb{R}^{n \times n}$ the $n$-dimensional real vector space and $n$-by-$n$ real matrix space, respectively. Analogously, we use $\mathbb{C}^n$ and $\mathbb{C}^{n \times n}$ to denote the corresponding complex vector and complex matrix spaces. We consider the following standard eigenvalue problem
\begin{align}
\label{equation-intr-EigeProb}
Ax=\lambda x, \quad \mbox{with} \quad \|x\|=1,
\end{align}
where $A\in\mathbb{R}^{n\times n}$ is a large, sparse and non-symmetric matrix, $(\lambda,x)$ with $\lambda\in\mathbb{C}$ and $x\in\mathbb{C}^{n}$ is an eigenpair of $A$, and $\|\cdot\|$ denotes the Euclidean norm of the corresponding vector or matrix. The eigenvalue problem (\ref{equation-intr-EigeProb}) is widely concerned in scientific and engineering computing, and in this paper we mainly aim to compute the eigenvalues with the largest real parts and their corresponding eigenvectors of non-symmetric matrix $A$. Several iteration methods for solving the symmetric (or non-symmetric) standard (or generalized) eigenvalue problem have been developed so far, such as gradient-type methods \cite{Knyazev-01,Ovtchinnikov-SIAM06,VecharynskiYP-15}, standard and rational Krylov subspace methods \cite{Parlett-98,Ruhe-98,Saad-11}, Davidson-type methods \cite{CrouzeixPS-94,Davidson-75,Morgan-92,SleijpenV-96,ZhouS-07}, and their variants \cite{BaiM-NA19,BaiW-20,Miao-JCAM18,Miao-NA19,Sorensen-92}, etc.

However, in the subsequent discussions, we only concern those iteration methods that integrate polynomial filtering techniques into the iteration process. Polynomial filtering technique for computing eigenvalues of symmetric or non-symmetric matrices is not a new idea, and it actually computes a vector with the form of $z=p(A)z_o$ by matrix-vector products, where $z_o$ is the current approximate eigenvector to the desired one, and $p(\cdot)$ is a polynomial function chosen to enhance the components of $z_o$ in the direction of desired eigenvectors while at the same time to damp those in the direction of undesired ones. Several acceleration ways have been developed so far.

In \cite{Saad-84}, Chebyshev polynomial filters are designed to accelerate the Arnoldi process for non-symmetric matrices, which is referred to as the {\em Arnoldi-Chebyshev} method. In this method, first Arnoldi procedure is performed to obtain some approximations to the desired eigenvector; then the initial vector for the Chebyshev iteration is computed, which is the current Ritz vector or a linear combination of the approximate eigenvectors obtained by the Arnoldi iteration; and, finally, the vector filtered by the Chebyshev iteration is used as a starting vector to restart the Arnoldi procedure. The block Arnoldi-Chebyshev method \cite{Sadkane-93} has been proposed to compute several eigenvalues with the largest real parts of the non-symmetric matrix.

The {\em filtered Lanczos} method proposed in \cite{BekasKS-08,FangS-12} is another application of polynomial filtering technique, in which the symmetric matrix $A$ is replaced by $p(A)$ to carry out the Lanczos procedure. More precisely, after $\ell$ steps of the Lanczos process with matrix $p(A)$ from an initial vector $v$, we obtain an orthonormal basis $V$ of the standard filtered Krylov subspace
\begin{align}
\label{equation-intr-FilKrySub}
\mathcal{K}_{\ell}(A,v)
=\text{span} \left\{v, \, p(A)v, \,p^{2}(A)v, \, \ldots, \, p^{\ell-1}(A)v \right\},
\end{align}
and, as a by-product, the projected matrix $T=V^Tp(A)V$ can be acquired without extra computations. The eigenvalues of matrix $T$ are computed to approximate the desired eigenvalues located in an interval $[\xi,\eta]$. The polynomial filter $p(\lambda)$ and another interval $[a, b]$, which contains the spectrum of $A$ (then $[\xi,\eta]\subset[a, b]$ holds), are prescribed in advance before the whole iteration process. A loose interval $[a, b]$ would influence and decrease the effectiveness of the polynomial filter, and the polynomial filter should be chosen such that the value $p(\lambda)$ in $[\xi,\eta]$ is extremely larger than that in $[a, b]\setminus[\xi,\eta]$. We should point out that this method fails to compute a specified eigenvalue, e.g., the smallest or the largest one, because $A$ and $p(A)$ do not share same eigenvalues. It indicates that we can not use the smallest (or the largest) eigenvalue of the projected matrix $T$ to approximate that of matrix $A$. However, this drawback can be resolved with a minor modification that projects matrix $A$ other than $p(A)$ onto the corresponding projection subspace. In the sequel, we will refer to the iteration process that projects matrix $A$ onto the filtered Krylov subspace in (\ref{equation-intr-FilKrySub}) as the filtered Krylov subspace method for both symmetric and non-symmetric eigenvalue problems.

Polynomial filtering technique is also efficient for accelerating convergence of Davidson method. In \cite{ZhouS-07}, Zhou and Saad proposed the {\em Chebyshev-Davidson} method for solving symmetric eigenvalue problems, in which the solution of the correction equation involved in the Davidson method was replaced by a Chebyshev polynomial filtered vector. However, to our best knowledge, the Chebyshev-Davidson method as well as the filtered Krylov subspace method has not appeared for solving non-symmetric eigenvalue problems so far.

Inspired by the existing polynomial filtering techniques, in this paper, we generalize the filtered Krylov subspace method and the Chebyshev-Davidson method, and propose a class of relaxed filtered Krylov subspace methods for non-symmetric eigenvalue problems. The relaxation of this method is twofold. On one hand, the polynomial filters are not fixed in the iteration process and vary from one Krylov to another. That is, the polynomial filters are adjusted dynamically step by step by utilizing the approximate eigenvalues obtained at each step, which can make the polynomial filtering process more effective and robust. On the other hand, a relaxed vector is inserted to guarantee a more accurate starting vector of the polynomial filtering process.

As we know, the Chebyshev polynomials show great superiority and robustness in polynomial acceleration techniques. In this paper, through exploring the convergence property of the complex Chebyshev polynomials, we also provide a simple but effective way to determine an ellipse, which affects the performance of the complex Chebyshev polynomial filters.

This paper is organized as follows. We present necessary preliminaries and notation in Section \ref{section-basics}. In Section \ref{section-rela}, we introduce a class of relaxed filtered Krylov subspaces and propose the corresponding relaxed filtered Krylov subspace method for solving non-symmetric eigenvalue problems. In Section \ref{section-cheb}, we present the convergence analysis and some further discussions of the complex Chebyshev iteration. In section \ref{section-nume}, some numerical experiments are executed to examine the competitiveness of our proposed method, i.e., the relaxed filtered Krylov subspace method, and, in the last section, we end this paper by some concluding remarks.

\section{Preliminaries and Notation}
\label{section-basics}

Throughout this paper, $\text{Re}(\lambda)$ and $\text{Im}(\lambda)$ represent the real and imaginary parts of a complex number $\lambda$, respectively. For a real or a complex number $\lambda$, its modulus is denoted by $|\lambda|$. $I_{n\times n}$ is used to denote the $n$-by-$n$ identity matrix, and without any confusions we just use $I_n$ or $I$ to simplify it. Notation $\mathcal{P}_m$ indicates the set of all polynomials with degree $m$. For the sake of convergence, matrix $A$ in this paper is assumed to be diagonalizable. The spectrum of $A\in\mathbb{R}^{n\times n}$ is denoted by $\{\lambda_i\}_{i=1}^n$ with decreasing order of their real parts, i.e., $\Re(\lambda_1)\geq\Re(\lambda_2)\geq\ldots\geq\Re(\lambda_n)$, and the corresponding unit eigenvectors are denoted by $\{x_i\}_{i=1}^n$.

Next, we give some basics of the Chebyshev polynomial filter which will be used later.

Assume that $x$ is the current approximate eigenvector and has an eigen-decomposition as $x=\mu_1x_1+\mu_2x_2+\cdots+\mu_nx_n$, then the polynomial filtered vector $z=p_m(A)x$ can also be expanded by the eigenbasis $\{x_i\}_{i=1}^n$ as
\begin{align}
\label{equation-cheb-FiltEigeDecom}
z=p_m(A)x=\mu_1p_m(\lambda_1)x_1+\sum\limits_{i=2}^n\mu_ip_m(\lambda_i)x_i.
\end{align}
In order to make $z$ be a good approximation to the desired eigenvector $x_1$, polynomial $p_m(\lambda)$ is often chosen such that the moduli of its values at $\{\lambda_i\}_{i=2}^n$, i.e., $|p_m(\lambda_i)|$, $i=2, 3, \ldots, n$, are much smaller than that at $\lambda_1$. A normalization condition $p_m(\lambda_1)=1$
imposed on the polynomial $p_m(\lambda)$ essentially leads us to find polynomials with their moduli being small on a discrete set $\{\lambda_2,\lambda_3,\ldots,\lambda_n\}$. However, without the knowledge of all eigenvalues of matrix $A$, it is almost impossible for us to find this class of polynomials. Thus, as an alternative, we attempt to replace the discrete set $\{\lambda_2,\lambda_3,\ldots,\lambda_n\}$ by a continuous domain $E$ which contains the discrete set $\{\lambda_2,\lambda_3,\ldots,\lambda_n\}$ but excludes $\lambda_1$. That is, we need to seek a polynomial $p_m(\lambda)$ which achieves the minimum
\begin{align*}
\min\limits_{p\in \mathcal{P}_{m},p(\lambda_1)=1}\max\limits_{\lambda\in E}|p(\lambda)|.
\end{align*}

As we know, the spectrum of the non-symmetric matrix $A\in\mathbb{R}^{n\times n}$ is symmetric with respect to the real axis, therefore, it is favourable to restrict domain $E$ to an ellipse $E(d,c,a)$, with real center $d$, focuses $d+c$ and $d-c$ and major semiaxis $a$, which is symmetric with respect to the real axis. Then, the required polynomial can be achieved by
\begin{align}
\label{equation-cheb-ChebPoly}
p_m(\lambda)=\frac{T_m[(\lambda-d)/c]}{T_m[(\lambda_1-d)/c]},
\end{align}
where $T_m(z)=\frac{1}{2}\left(w^m+w^{-m}\right)$, with $w$ being the modulus largest number which satisfies $z=\frac{1}{2}\left(w+w^{-1}\right)$, is the complex Chebyshev polynomial with degree $m$ of the first kind. Note that, by technical modifications, the polynomial filter in (\ref{equation-cheb-ChebPoly}) benefits from the three-term recurrence relation
\begin{align*}
p_{m+1}(\lambda)=2\sigma_{m+1}\frac{\lambda-d}{c}p_m(\lambda)
-\sigma_m\sigma_{m+1}p_{m-1}(\lambda),
\end{align*}
with $\sigma_{m+1}=\frac{1}{2/\sigma_1-\sigma_m}$, $m=1,2,\ldots$, and $\sigma_1=\frac{c}{\lambda_1-d}$.

For convenience, we present an algorithmic description of the Chebyshev iteration as follows. For more details about the Chebyshev iteration, we refer to \cite{Saad-84,Saad-11} and the references therein.

\begin{method}{\bf(Chebyshev Iteration)}
\label{algorithm-cheb-ChebIter}
{\sf
\begin{enumerate}
\item {\bf Start:}
Given the non-symmetric matrix $A$, the starting vector $z_0$, and the constants $d$, $c$ and $\lambda_1$; Compute $\sigma_1=\frac{c}{\lambda_1-d}$ and $z_1=\frac{\sigma_1}{c}(A-dI)z_0$.

\item {\bf Iterate:}
For $k=1,2,\ldots$, until convergence do:
\begin{enumerate}
\item
compute $\sigma_{k+1}=\frac{1}{2/\sigma_1-\sigma_k}$;
\item
compute $z_{k+1}=2\frac{\sigma_{k+1}}{c}(A-dI)z_k-\sigma_k\sigma_{k+1}z_{k-1}$.
\end{enumerate}
\end{enumerate}
}
\end{method}

Note that although $a$ and $c$ are either real or purely imaginary numbers, the recursion in the above method can still be carried out in real arithmetic, because the scalars $\frac{\sigma_{k+1}}{c}$ and $\sigma_k\sigma_{k+1}$ are real numbers if $\lambda_1$ and $z_0$ are both real.

\section{The Relaxed Filtered Krylov Subspace Method}
\label{section-rela}
In this section, we aim to generalize the filtered Lanczos method for computing the eigenvalues located in a prescribed interval of symmetric matrices to a more general form for computing specified eigenvalues, e.g., the eigenvalues with the algebraically largest real parts, of non-symmetric matrices.

As afore-mentioned, in the filtered Lanczos method, the polynomial filter is prescribed in advance and fixed in the whole iteration process, thus an elaborate interval containing the spectrum of matrix $A$ should be provided to make the polynomial filter very effective, which is cumbersome and may be difficult in some circumstances. In fact, one way to improve the filtered Lanczos method is to vary the involved polynomial filters from one Krylov iteration to another, so that we can adequately utilize the currently obtained approximate eigenvalues. Moreover, since the starting vector for the polynomial filtered process also plays an important role, another reasonable strategy to improve the filtered Lanczos method is to choose the starting vector for the polynomial filtered process as a linear combination of the orthonormal vectors, which span the current projection subspace.

Therefore, with the starting unit vector $v_1$, to extract the desired eigenvalues and the corresponding eigenvectors, we project matrix $A$ onto the following relaxed filtered Krylov subspace
\begin{align}
\label{equation-rela-RelaFilKrySub}
\mathcal{K}_{\ell}^{(R)}(A,v_1)
=\text{span} \left\{v_1, \, p^{(1)}(A)V_1s_1, \,
p^{(2)}(A)V_2s_2, \, \ldots, \, p^{(\ell-1)}(A)V_{\ell-1}s_{\ell-1} \right\},~\mbox{for}~\ell \geq 2,
\end{align}
where for $j\in\{1,2,\ldots,\ell-1\}$, $V_j$ is a matrix with orthonormal columns spanning the $j$-th projection subspace $\mathcal{K}_{j}^{(R)}(A,v_1)$, $s_j\in\mathbb{R}^j$ is a prescribed vector, and $p^{(j)}(\lambda)=a_0+a_1\lambda+\cdots
+a_{m_j}\lambda^{m_j}$ is a constructed polynomial filter with its degree $m_j$ larger than zero. Obviously, if all polynomials $p^{(j)}(\lambda)$, $j=1,2,\ldots,\ell-1$, are fixed as a same polynomial $p(\lambda)$, and $s_j$, $j\in\{1,2,\ldots,\ell-1\}$, is chosen to be $s_j=(0,\ldots,0,1)^T\in\mathbb{R}^j$, the relaxed filtered Krylov subspace in (\ref{equation-rela-RelaFilKrySub}) will degenerate into the standard filtered Krylov subspace in (\ref{equation-intr-FilKrySub}).

Based on the relaxed filtered Krylov subspace in (\ref{equation-rela-RelaFilKrySub}), we propose the relaxed filtered Krylov subspace method for computing the eigenvalues with the largest real parts, which is summarized algorithmically as follows. We remark that with minor modifications, the proposed method can also be suitable for computing the eigenvalues with the smallest real parts.
\begin{method}{\bf (The Relaxed Filtered Krylov Subspace Method)}
\label{method-rela-RelFilKrySub}
{\sf
\begin{enumerate}
\item {\bf Start:}
Choose a unit approximate eigenvector $x^{(1)}\in\mathbb{R}^{n}$ and set $v_1=x^{(1)}$; Compute $w_1=Av_1$ and $H_1=v_1^Tw_1$, respectively; Set $V_1=[v_1]$, $W_1=[w_1]$, $s_1=1$ and $p^{(1)}(\lambda)=\lambda$; Let $\varepsilon$ be the stopping tolerance and $n_r$ be the restart number.

\item {\bf Iterate:}
For $k=1,2,\ldots$, do:
\begin{enumerate}

\item
compute the polynomial filtered vector $z=p^{(k)}(A)V_ks_k$, orthonormalize $z$ against $V_k$ via modified Gram-Schmidt process to obtain $v_{k+1}$, i.e., $v_{k+1}=\frac{(I-V_kV_k^T)z}{\|(I-V_kV_k^T)z\|}$, and set $V_{k+1}=[V_k,v_{k+1}]$;

\item
compute $w_{k+1}=Av_{k+1}$, $W_{k+1}=[W_k,w_{k+1}]$, $H_{k,u}=V_k^Tw_{k+1}$, $H_{k,l}=v_{k+1}^TW_k$ and $h_{k+1,k+1}=v_{k+1}^Tw_{k+1}$, respectively;

\item
form the projected matrix
$H_{k+1}=
\left(
  \begin{array}{cc}
    H_k & H_{k,u} \\
   H_{k,l} & h_{k+1,k+1}
  \end{array}
\right);$

\item
compute eigenpairs $(\theta_i,y_i)$ of the projected matrix $H_{k+1}$, with descending order $\text{Re}(\theta_1)\geq \text{Re}(\theta_2)\geq\ldots\geq \text{Re}(\theta_{k+1})$, that is, $H_{k+1}\, y_i=\theta_i y_i$ with $\|y_i\|=1$, $i=1,2,\ldots,k+1$;

\item
compute the corresponding Ritz vector $x^{(k+1)}=V_{k+1}\, y_1$, compute the residual $r^{(k+1)}=W_{k+1}y_1-\theta_1 x^{(k+1)}$, and test for convergence: if $\|r^{(k+1)}\|\leq\varepsilon$ is satisfied, then stop;

\item
if $k=n_r-1$, then goto Step 3 to restart;
\item
use the obtained approximations to construct an appropriate polynomial filter $p^{(k+1)}(\lambda)$;

\item
choose a proper vector $s_{k+1}\in\mathbb{R}^{k+1}$, and set $k:=k+1$.
\end{enumerate}
\item {\bf Restart:}
Let $v_1=x^{(n_r)}$, compute $w_1=Av_1$ and $H_1=v_1^Tw_1$, set $V_1=[v_1]$, $W_1=[w_1]$, $s_1=1$ and $p^{(1)}(\lambda)=\lambda$, and goto Step 2.
\end{enumerate}
}
\end{method}
\vspace{7pt}

In the following, we make some detailed comments and implementations on Method \ref{method-rela-RelFilKrySub} for solving non-symmetric eigenvalue problems.

Steps {\sf 2.(a)-(d)} are called the Rayleigh-Ritz procedure for projecting matrix $A$ onto the projection subspace in (\ref{equation-rela-RelaFilKrySub}). In this process, we first construct an orthonormal basis $V_{k+1}$ for the projection subspace; then we project matrix $A$ onto the corresponding subspace and form the associated projected matrix $H_{k+1}=(V_{k+1})^TAV_{k+1}$; and, finally, we use the eigenvalues with the largest real parts of the projected matrix to approximate the desired ones.

For the orthogonalization process in Step {\sf 2.(a)}, we use the modified Gram-Schmidt method. As the iteration proceeds, the orthogonalization of $v_{k+1}$ to $V_k$ may be lost, and at this time the reorthogonalization process is recommended.

Generally speaking, when the dimension of the projection subspace is large enough, which would result in large storage and high computational costs, we need to adopt a restart strategy. Step 3 provides us one common way that takes the last obtained approximate eigenvector to restart. We remark that this may not be the most efficient way for restarting, as we may discard possibly valuable information contained in the projection subspace. Thereby, a combination of the current Ritz vectors also provides us an alternative strategy for restarting.

In Step {\sf 2.(g)}, the selection of proper polynomial filters is one of the most critical ingredients of the relaxed filtered Krylov subspace method. For symmetric eigenvalue problems, many polynomial filters or rational function approximations to step functions are available, see, e.g., \cite{Anderson-10,FangS-12,Miao-EAJAM17,XiS-16} and the references therein. Unfortunately, to the best of our knowledge, for non-symmetric eigenvalue problems rare polynomial filters or rational function approximations have been developed to accelerate the convergence. Thus, in this paper we mainly focus on the Chebyshve polynomial filters and the detailed discussions will be given in the next section.

The relaxed filtered Krylov subspace method may consume more computational costs than the filtered Krylov subspace method for extra computations of $s_k$ in Step {\sf 2.(h)} and $V_ks_k$, but the reduction in total number of iteration steps will bring us a great advantage in terms of total computational cost and computing time. The ways for selecting the vector $s_k$ in Method \ref{method-rela-RelFilKrySub} include but are not limited to those in the following.

First, as the simplest means, vector $s_k$ can be determined by $s_k=(0,\ldots,0,1)^T\in\mathbb{R}^{k}$. For this selection, the starting vector $V_ks_k$ of the polynomial filtered process is the last obtained basis vector $v_k$, then Method \ref{method-rela-RelFilKrySub} is reduced to the filtered Lanczos method with minor modifications if the target matrix $A$ is symmetric and all polynomials are fixed to a same one. Based on this selection, we can construct the corresponding filtered Krylov subspace method for solving non-symmetric eigenvalue problems.

Second, as an intuitive generalization of the first choice of vector $s_k$, it can be chosen as a weight vector, that is, $s_k=(\alpha_1,\alpha_2,\ldots,\alpha_k)^T$ with $\alpha_i\geq 0$, $i=1, 2, \ldots, k$, and $\sum\limits_{i=1}^{k}\alpha_i=1$. As the latter basis vectors in $V_k$, e.g., $v_k, v_{k-1}, v_{k-2}, \ldots$, may contain more valuable approximate information than the former ones, e.g., $v_1, v_2, v_3, \ldots$, the last $\alpha_i$'s can be chosen to be appropriately larger than the former ones.

Third, the relaxed filtered Krylov subspace method will degenerate to a polynomial filtered Davidson method if the starting vector $V_ks_k$ of the polynomial filtered process is chosen to be the currently obtained Ritz vector $x^{(k)}$. That is, vector $s_k$ can be prescribed to be the eigenvector $y_1$ corresponding to the eigenvalue $\theta_1$ whose real part is the largest among all the eigenvalues of the projected matrix $H_{k}$. In this case, the relaxed filtered Krylov subspace method can be regarded as a generalization for the Chebyshev-Davidson method \cite{Zhou-10,ZhouS-07} from symmetric to non-symmetric eigenvalue problems.

Finally, we can expect $V_ks_k$, which is chosen from the current projection subspace spanned by $V_k$, to minimize the residual vector associated with the desired eigenvalue $\lambda_1$, that is,
\begin{align}
\label{equation-rela-RelVecS}
s_k=\mathop{\arg\min}_{s\in\mathbb{R}^{k},\|s\|=1}\|(A-\lambda_1I)V_ks\|.
\end{align}
For the unknown of the desired eigenvalue $\lambda_1$, we often replace it with an approximation, say, the Ritz value $\theta_1$ at the $k$-th step. In \cite{Jia-00}, Jia referred to the vector $s_k$ obtained in this way as a refined Ritz vector. Obviously, $s_k$ is the right singular vector corresponding to the smallest singular value of matrix $\widehat{W}_k=(A-\theta_1I)V_k$. Thus, we can directly form $\widehat{H}_k=V_k^T(A-\theta_1I)^T(A-\theta_1I)V_k
=W_k^TW_k-\theta_1W_k^TV_k-\theta_1V_k^TW_k+\theta_1^2I$ to gain the right singular vector, or make a QR factorization $(A-\theta_1 I)V_k=QR$ and compute the singular vectors of the small triangular matrix $R$. The process for computing a singular vector looks cumbersome, but, in practice, the total extra cost is negligible compared to the whole iteration cost. We remark that the matrix $\widehat{H}_k$ can be updated just by updating $W_k^TW_k$ and $W_k^TV_k$ step by step as $V_k^TW_k=H_k$.

\section{Chebyshev Filtered Iteration}
\label{section-cheb}
As we know, the Chebyshev polynomial is effective and economical when employed to solve large symmetric or non-symmetric eigenvalue problems, because it can be used to highly amplify the desired components and filter out the undesired directions, and at the same time, it avoids matrix factorizations in shift-and-invert technique by elegantly performing matrix-vector products.

This section is devoted to implementing the polynomial filtered process with Chebyshev polynomials and providing convergence properties of the Chebyshev iteration.

\subsection{Convergence}
To explore the convergence of the Chebyshev iteration, it can be seen from the equality in (\ref{equation-cheb-FiltEigeDecom}) that we need to give an analysis for the quantity $p_m(\lambda_i)$ in front of $x_i$, $i=2, 3, \ldots, n$. According to the definition of the complex Chebyshev polynomial in (\ref{equation-cheb-ChebPoly}), $p_m(\lambda_i)$ can be rewritten as
\begin{align}
\label{equation-cheb-ChebPoly-i}
p_m(\lambda_i)=\frac{w_i^m+w_i^{-m}}{w_1^m+w_1^{-m}},
\end{align}
where $w_i$ is the modulus largest root of the equation
\begin{align}
\label{equation-cheb-RootEqu}
\frac{1}{2}\left(w+w^{-1}\right)=\frac{\lambda_i-d}{c}
\end{align}
of the variable $w$, $i=1, 2, \ldots, n$. By straightforward computations, we can see that $w_i$ is the modulus larger one of
\begin{align*}
w_i^+=\frac{\lambda_i-d}{c}
+\sqrt{\left(\frac{\lambda_i-d}{c}\right)^2-1}
\quad\text{and}\quad
w_i^-=\frac{\lambda_i-d}{c}-\sqrt{\left(\frac{\lambda_i-d}{c}\right)^2-1}.
\end{align*}
In fact, $w_i^+$ and $w_i^-$ are the inverses of each other, i.e., $w_i^+w_i^-=1$. Thus, we have
\begin{align*}
p_m(\lambda_i)
=\frac{\left(w_i^+\right)^m+\left(w_i^+\right)^{-m}}{w_1^m+w_1^{-m}}
=\frac{\left(w_i^-\right)^m+\left(w_i^-\right)^{-m}}{w_1^m+w_1^{-m}},
~i=2, 3, \ldots, n,
\end{align*}
which indicates that choosing $w_i^+$ or $w_i^-$ for $w_i$ in expression (\ref{equation-cheb-ChebPoly-i}) makes no difference on the value of $p_m(\lambda_i)$.

For any $z\in\mathbb{C}$, we use $\sqrt{z}$ to denote its {\em arithmetic} square root, i.e., the one of its square roots with positive real part, or with zero real part and non-negative imaginary part. This definition is in accordance with the arithmetic square root of a non-negative real number. Besides, this definition for a complex number makes no difference on the value of $w_i, i=1, 2, \ldots, n$, and then makes no difference on the value of $p_m(\lambda_i), i=1, 2, \ldots, n$. In fact, if $w_i^+$ ($w_i^-$) corresponds to the modulus larger root $w_i$ of equation (\ref{equation-cheb-RootEqu}) under the above definition, then when $\sqrt{z}$ denotes the other square root of $z\in\mathbb{C}$, i.e., the one with negative real part, or with zero real part and non-positive imaginary part, the modulus larger root $w_i$ of the equation (\ref{equation-cheb-RootEqu}) will be $w_i^-$ ($w_i^+$), and no matter which definition of $\sqrt{z}$ for $z\in\mathbb{C}$ is chosen, the actual value of $w_i$ does not change. This is the reason why we use $\sqrt{z}$ to denote its arithmetic square root.

Under the definition of $\sqrt{z}$ for any $z\in\mathbb{C}$ as above, we present the subsequent lemma
whose proof is left to Section \ref{sect-append}.
\begin{lemma}
\label{lemma-cheb-inequalities}
For any complex number $z\in\mathbb{C}$, the following two assertions hold true:
\begin{enumerate}
\item[(i)]
if the real and imaginary parts of $z$ satisfy $\Re(z)>0$, or satisfy $\Re(z)=0$ and $\Im(z) \geq 0$, then
\begin{align*}
\left|z-\sqrt{z^2-1}\right|
\leq \left||z|+\sqrt{|z|^2-1}\right|
\leq \left|z+\sqrt{z^2-1}\right|
\leq |z|+\sqrt{|z|^2+1};
\end{align*}
\item[(ii)]
if the real and imaginary parts of $z$ satisfy $\Re(z)<0$, or satisfy $\Re(z)=0$ and $\Im(z) \leq 0$, then
\begin{align*}
\left|z+\sqrt{z^2-1}\right|
\leq \left||z|+\sqrt{|z|^2-1}\right|
\leq \left|z-\sqrt{z^2-1}\right|
\leq |z|+\sqrt{|z|^2+1}.
\end{align*}
\end{enumerate}
\end{lemma}

Based on the above discussions, we know that when the real and imaginary parts of $\frac{\lambda_i-d}{c}$ satisfy $\Re(\frac{\lambda_i-d}{c})>0$, or satisfy $\Re(\frac{\lambda_i-d}{c})=0$ and $\Im(\frac{\lambda_i-d}{c})\geq0$, the modulus largest root $w_i$ can be prescribed with $w_i=w_i^+$; when the real and imaginary parts of $\frac{\lambda_i-d}{c}$ satisfy $\Re(\frac{\lambda_i-d}{c})<0$, or satisfy $\Re(\frac{\lambda_i-d}{c})=0$ and $\Im(\frac{\lambda_i-d}{c}) \leq 0$, the modulus largest root $w_i$ can be prescribed with $w_i=w_i^-$, $i=1, 2, \ldots, n$. In addition, $|w_i|$, $i=1, 2, \ldots, n$, can be bounded by the following inequalities
\begin{align}
\label{equation-cheb-OmegaBound}
\left|\left|\frac{\lambda_i-d}{c}\right|
+\sqrt{\left|\frac{\lambda_i-d}{c}\right|^2-1}\right|
\leq |w_i|
\leq \left|\frac{\lambda_i-d}{c}\right|
+\sqrt{\left|\frac{\lambda_i-d}{c}\right|^2+1}.
\end{align}

Next, we give the analysis of the upper bound of $\left|p_m(\lambda_i)\right|,
i=2, 3, \ldots, n$. It follows from (\ref{equation-cheb-ChebPoly-i}) that
\begin{align*}
p_m(\lambda_i)=\frac{w_i^m+w_i^{-m}}{w_1^m+w_1^{-m}}
=\frac{w_i^m}{w_1^m}\cdot\frac{1+w_i^{-2m}}{1+w_1^{-2m}}.
\end{align*}
Since $w_i^+w_i^-=1$ and $w_i$ is the modulus larger one of $w_i^+$ and $w_i^-$, we have $|w_i|\geq 1, i=1, 2, \ldots, n$. Then, it holds that $\frac{1+w_i^{-2m}}{1+w_1^{-2m}}\rightarrow 1$ as $m\rightarrow\infty$ if $|w_i|\neq 1$ and $|w_1|\neq 1$, and $\frac{1+w_i^{-2m}}{1+w_1^{-2m}}$ is bounded if $|w_i|=1$ or $|w_1|=1$, thus, we define ${\kappa}_i=\left|\frac{w_i}{w_1}\right|$ as the damping coefficient \cite{Saad-84} of the eigenvalue $\lambda_i$, $i=2, 3, \ldots, n$. The convergence of the Chebyshev iteration is determined by the maximum damping coefficient $\kappa$, i.e., ${\kappa}=\max\limits_{2\leq i\leq n}\{\kappa_i\}$.

As we know, the eigenvalues of non-symmetric matrices are real or conjugate complex numbers, that is, all the eigenvalues are symmetric with respect to real axis. Therefore, for the ellipse $E(d, c, a)$ determined by the Chebyshev iteration, which contains the undesired eigenvalues $\{\lambda_i\}_{i=2}^n$ but excludes the desired eigenvalue $\lambda_1$, the major axis is either on real axis or parallel to imaginary axis; see Figure \ref{fig:ellipse-fat-thin}. For convenience, we refer to the ellipse whose major axis is on real axis as a ``fat" ellipse, while the ellipse whose major axis is parallel to imaginary axis as a ``thin" ellipse.

\begin{figure}[ht]
\centering \subfigure{
\begin{minipage}[b]{0.45\textwidth}
\centering
\includegraphics[scale=0.45]{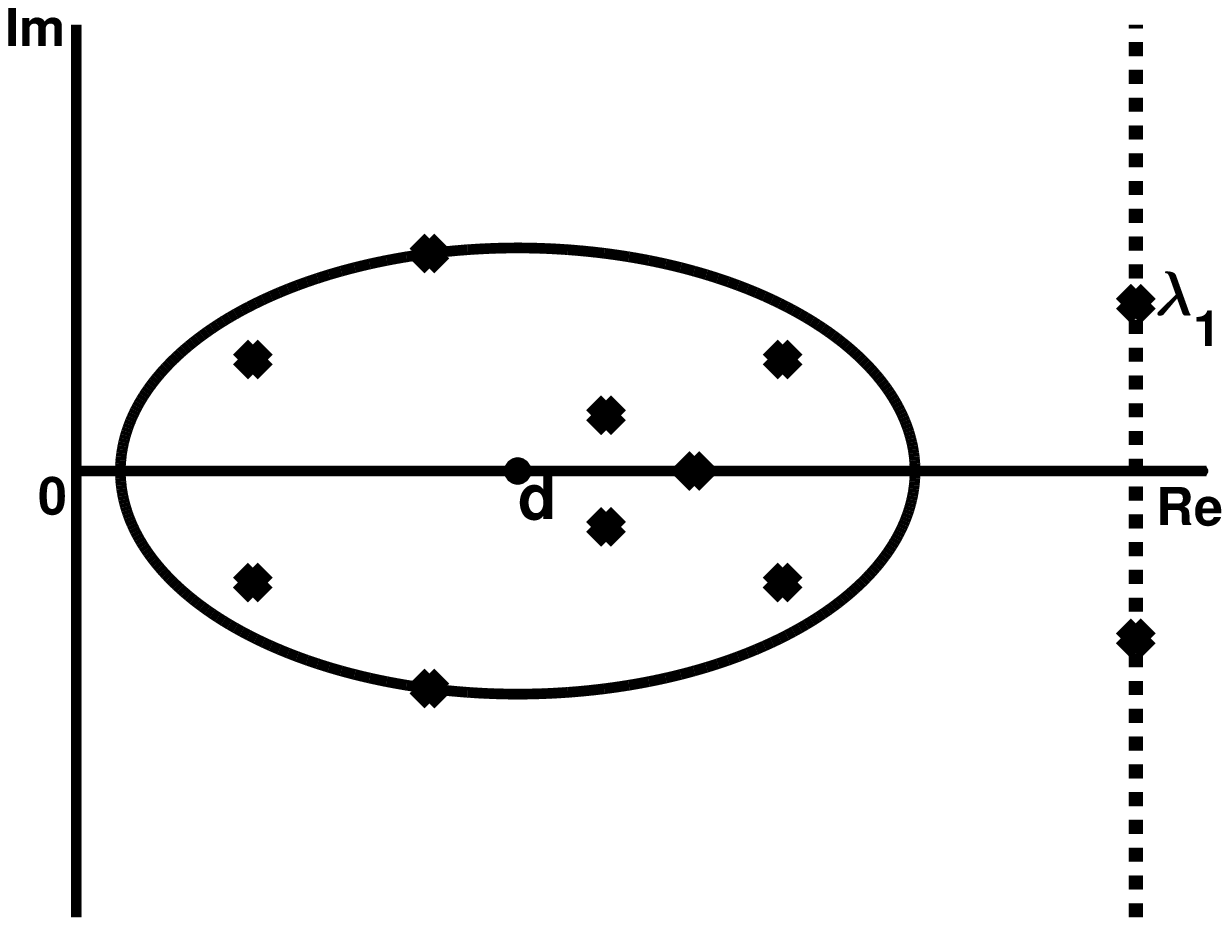}
\end{minipage}
}
\centering \subfigure{
\begin{minipage}[b]{0.45\textwidth}
\centering
\includegraphics[scale=0.45]{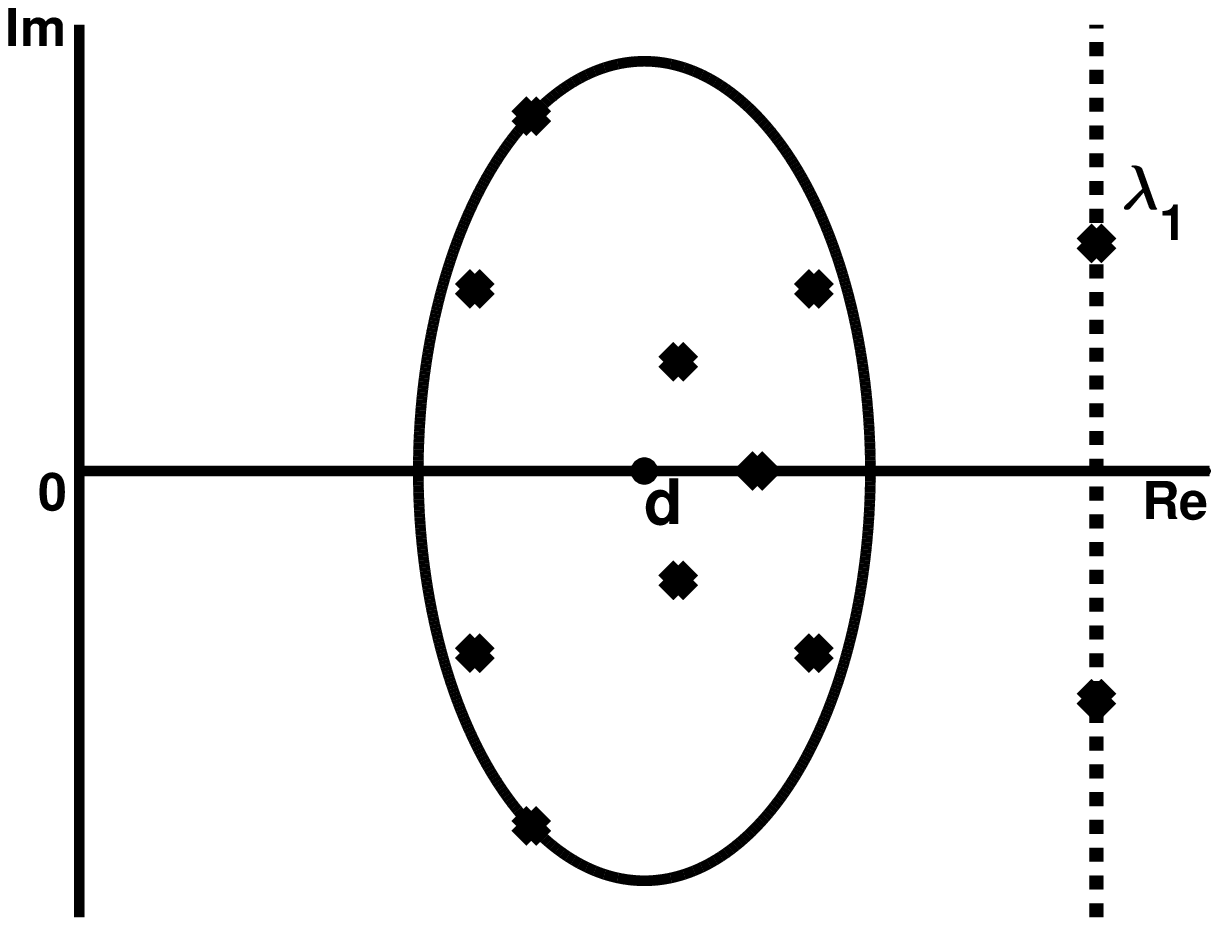}
\end{minipage}
}
\caption{Pictures of fat (left) and thin (right) ellipses
determined by Chebyshev iteration.}
\label{fig:ellipse-fat-thin}
\end{figure}

Now, we can establish convergence of the Chebyshev iteration.
\begin{theorem}
\label{theorem-cheb-Conv}
If the ellipse $E(d, c, a)$ determined by the Chebyshev iteration, with $d$ being the ellipse center, $d-c$ and $d+c$ being the foci, and $a$ being the major semiaxis, contains the undesired eigenvalues $\{\lambda_i\}_{i=2}^n$, then the maximum damping coefficient ${\kappa}$ generated by the Chebyshev iteration satisfies
\begin{align}
\label{equation-cheb-TheoremIneq}
\kappa\leq
\frac{|a|+\sqrt{|a|^2+|c|^2}}
{\left||\lambda_1-d|+\sqrt{|\lambda_1-d|^2-|c|^2}\right|}.
\end{align}
\end{theorem}

\begin{proof}
Since the ellipse $E(d, c, a)$ contains the eigenvalues $\{\lambda_i\}_{i=2}^n$, we have $\left|\frac{\lambda_i-d}{c}\right|\leq \left|\frac{a}{c}\right|$, $i=2, 3, \ldots, n$.
According to the estimates in (\ref{equation-cheb-OmegaBound}), we have
\begin{align*}
|w_i|
\leq \left|\frac{\lambda_i-d}{c}\right|
+\sqrt{\left|\frac{\lambda_i-d}{c}\right|^2+1}
\leq \left|\frac{a}{c}\right|+\sqrt{\left|\frac{a}{c}\right|^2+1},
\quad i=2, 3, \ldots, n,
\end{align*}
and
\begin{align*}
|w_1|
\geq \left|\left|\frac{\lambda_1-d}{c}\right|
+\sqrt{\left|\frac{\lambda_1-d}{c}\right|^2-1}\right|.
\end{align*}
Therefore, the maximum damping coefficient $\kappa$ satisfies
\begin{align*}
\kappa
= \max_{2\leq i\leq n} {\kappa}_i
= \max_{2\leq i\leq n} \left|\frac{w_i}{w_1}\right|
& \leq \frac{\left|\frac{a}{c}\right|+\sqrt{\left|\frac{a}{c}\right|^2+1}}
{\left|\left|\frac{\lambda_1-d}{c}\right|
+\sqrt{\left|\frac{\lambda_1-d}{c}\right|^2-1}\right|} \\
& = \frac{|a|+\sqrt{|a|^2+|c|^2}}
{\left||\lambda_1-d|+\sqrt{|\lambda_1-d|^2-|c|^2}\right|},
\end{align*}
which leads to the validity of estimate (\ref{equation-cheb-TheoremIneq}).
\end{proof}

If the imaginary parts of the undesired eigenvalues of matrix $A$ are relatively small, that is, the ellipse $E(d, c, a)$ determined by the Chebyshev iteration is fat, like the left picture in Figure \ref{fig:ellipse-fat-thin}, it is beneficial to compute the eigenvalues with the largest real parts. In fact, if $\lambda_1$ is well separated from the others $\{\lambda_i\}_{i=2}^{n}$, say, the real or imaginary parts of $\lambda_1$ is very large, which also indicates that $|\lambda_1-d|$ is relatively large compared with $a$, we can chose proper parameters $a$ and $c$ to ensure the maximum damping coefficient $\kappa$ as small as possible.

However, for the case that the imaginary parts of the undesired eigenvalues of matrix $A$ are dominant, if the ellipse $E(d, c, a)$ determined by the Chebyshev iteration is thin, like the right picture in Figure \ref{fig:ellipse-fat-thin}, the upper bound of $\kappa$ in Theorem \ref{theorem-cheb-Conv} may be very large as $|a|>|\lambda_1-d|$ may hold. This is disadvantageous for us to compute the eigenvalues with the largest real parts but advantageous for us to compute the eigenvalues with the largest imaginary parts. Therefore, for computing the eigenvalues with the largest real parts in this case, in the actual realizations of the proposed method, we prefer to determine an ellipse which is closer to a fat one. We will discuss how to determine the ellipse in the next subsection.

\subsection{Determining an Ellipse}
In this section, we consider how to determine a fat ellipse $E(d,c,a)$ to compute the eigenvalues with largest real parts. Denote by $\{\lambda_i\}_{i=r+1}^{n}$ the unwanted eigenvalues, and by $x_{+}=\max\limits_{r+1\leq i\leq n}\text{Re}(\lambda_i)$, $x_{-}=\min\limits_{r+1\leq i\leq n}\text{Re}(\lambda_i)$ and $y_{+}=\max\limits_{r+1\leq i\leq n}\text{Im}(\lambda_i)$. Without loss of generality, we suppose that $r=1$. Let $\zeta$ be a real number in the interval $(x_+, \lambda_1)$ when $\lambda_1$ is real, otherwise, it can be replaced by $\text{Re}(\lambda_1)$ when the imaginary part of $\lambda_1$ is not zero.

When $y_{+}^2<(\zeta-x_{+})(\zeta-x_{-})$, i.e., the imaginary parts of the undesired eigenvalues of matrix $A$ are relatively small, let $d=\frac{x_{+}+x_{-}}{2}$ and the ellipse $E(d,c,a)$ pass through the point $(x_{+},y_{+})$ to guarantee all the undesired eigenvalues being in the ellipse. Then the center $d$, the major semiaxis $a$ and the minor semiaxis $b$ of this ellipse satisfy
\begin{align}
\label{equation-furt-ElliExpr-fat}
\frac{(x_{+}-d)^2}{|a|^2}+\frac{y_{+}^2}{|b|^2}=1.
\end{align}
Note that condition $y_{+}^2<(\zeta-x_{+})(\zeta-x_{-})$ results in $\sqrt{(x_+-d)^2+y_+^2}<|\lambda_1-d|$, thus, we suppose relation $\sqrt{(x_+-d)^2+y_+^2}\leq |a|<|\lambda_1-d|$ holds true to guarantee $|b|\leq |a|$ and the desired eigenvalue $\lambda_1$ being outside of it.

Then, if $|a|^2-(x_{+}-d)^2=0$, we can obtain the expression of $a$ as $|a|=x_{+}-d$, and the upper bound $\kappa_u(|c|)$ of the damping coefficient $\kappa$ can be rewritten as
\begin{align*}
\kappa_u(|c|):=\frac{|a|+\sqrt{|a|^2+|c|^2}}
{\left||\lambda_1-d|+\sqrt{|\lambda_1-d|^2-|c|^2}\right|}
=\frac{x_{+}-d+\sqrt{(x_{+}-d)^2+|c|^2}}
{|\lambda_1-d|+\sqrt{|\lambda_1-d|^2-|c|^2}}.
\end{align*}
Thus, $\kappa_u(|c|)$ reaches its minimum $\kappa_u(|c_o|)=\frac{x_{+}-d}{|\lambda_1-d|}<1$ at $c_o=0$. If $|a|^2-(x_{+}-d)^2\neq0$, from equality (\ref{equation-furt-ElliExpr-fat}), we can obtain the expression of $b$ with respect to $a$ as $|b|=\frac{|a| \, y_{+}}{\sqrt{|a|^2-(x_{+}-d)^2}}$. Combining this expression of $b$ and the fact $|a|^2=|b|^2+|c|^2$, the upper bound $\kappa_u(|a|)$ of the damping coefficient $\kappa$ can be rewritten as
\begin{align*}
\kappa_u(|a|):=\frac{|a|+\sqrt{|a|^2+|c|^2}}
{\left||\lambda_1-d|+\sqrt{|\lambda_1-d|^2-|c|^2}\right|}
=\frac{|a|+\sqrt{2|a|^2-\frac{|a|^2y_+^2}{|a|^2-(x_+-d)^2}}}
{|\lambda_1-d|+\sqrt{|\lambda_1-d|^2-|a|^2+\frac{|a|^2y_+^2}{|a|^2-(x_+-d)^2}}}.
\end{align*}
Note that $\kappa_u(|a|)$ is a monotonically increasing function in interval $\left[\sqrt{(x_+-d)^2+y_+^2}, |\lambda_1-d|\right)$, thus, the minimum of $\kappa_u(|a|)$ is achieved at $|a_o|=\sqrt{(x_+-d)^2+y_+^2}$. By straight computations, we obtain $\kappa_u(|a_o|)=\frac{\sqrt{(x_+-d)^2+y_+^2}}{|\lambda_1-d|}<1$. At this time, the corresponding parameters $|b_o|=\frac{|a_o| \, y_{+}}{\sqrt{|a_o|^2-(x_{+}-d)^2}}=|a_o|$ and $|c_o|=\sqrt{|a_o|^2-|b_o|^2}=0$ can be obtained easily.

When $y_{+}^2\geq(\zeta-x_{+})(\zeta-x_{-})$, i.e., the imaginary parts of the undesired eigenvalues of matrix $A$ are relatively large, let the fat ellipse $E(d,c,a)$ pass through the points $(x_{+},y_{+})$ and $(\zeta,0)$. Then the center $d$, the major semiaxis $a$ and the minor semiaxis $b$ of this ellipse satisfy equality (\ref{equation-furt-ElliExpr-fat}) and $|a|^2=(\zeta-d)^2$. Suppose the relation $d\leq\frac{\zeta^2-x_{+}^2-y_{+}^2}{2(\zeta-x_{+})}$ holds true, since $\frac{\zeta^2-x_{+}^2-y_{+}^2}{2(\zeta-x_{+})}\leq\frac{x_{+}+x_{-}}{2}$, we can obtain $|b|=\frac{|a| \, y_{+}}{\sqrt{|a|^2-(x_{+}-d)^2}}\leq |a|=\zeta-d$ and all the undesired eigenvalues are contained in the ellipse. The area of the ellipse
\begin{align*}
S(d)=\frac{\pi |a|^2y_{+}}{\sqrt{|a|^2-(x_{+}-d)^2}}=\frac{\pi(\zeta-d)^2y_{+}}{\sqrt{(\zeta-d)^2-(x_{+}-d)^2}}
\end{align*}
is decreasing with respect to $d$, thus the area of the ellipse $S(d)$ reaches its minimum at $d_o=\frac{\zeta^2-x_{+}^2-y_{+}^2}{2(\zeta-x_{+})}$. At this time, the corresponding parameters $|a_o|=|b_o|=\frac{(\zeta-x_{+})^2+y_{+}^2}{2(\zeta-x_{+})}$ and $c_o=0$ can be obtained easily. By straight computations, we can obtain the corresponding upper bound of the damping coefficient $\kappa_u=\frac{|a_o|}{|\lambda_1-d_o|}<1$.

Next, we use Figure \ref{figure-cheb-Ellipse} to depict the fat ellipse determined by the above process.
\begin{figure}[ht]
\centering
\setlength{\unitlength}{1cm}
\begin{minipage}[t]{6cm}
\begin{picture}(6,4.5)
\includegraphics[width=6cm,height=4.5cm]{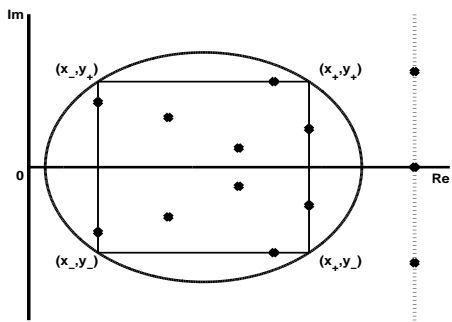}
\end{picture}\par
\end{minipage}\hfill
\caption{\label{figure-cheb-Ellipse}
Ellipse containing undesired approximate eigenvalues}
\end{figure}

\section{Numerical Experiments}
\label{section-nume}
In this section, by using the Chebyshev polynomial with the ellipse determined in Section~\ref{section-cheb}, we examine the numerical behavior of the {\em relaxed filtered Krylov subspace} ({\bf RFKS}) method, and compare it with the {\em filtered Krylov subspace} ({\bf FKS}) method, the {\em Chebyshev-Davidson} ({\bf CD}) method and the {\em Arnoldi-Chebyshev} ({\bf AC}) method in terms of the count of iteration steps ({``IT''}), the number of matrix-vector products ({``MV''}) and the computing time in seconds ({``CPU''}) for computing the eigenvalues with the largest real parts of the non-symmetric eigenvalue problem (\ref{equation-intr-EigeProb}).

Recall that, with an initial vector $v$ and a fixed polynomial filter $p(\lambda)$, the FKS method projects matrix $A$ instead of $p(A)$ onto the filtered Krylov subspace (\ref{equation-intr-FilKrySub}), and uses eigenvalues of the corresponding projected matrices to approximate the desired ones. The polynomial filter $p(\lambda)$ in FKS method should be prescribed in advance, and in our numerical experiments we execute 20 iteration steps of the Arnoldi method to determine it. The framework of CD method for non-symmetric matrices is similar to that of CD method for symmetric matrices proposed in \cite{Zhou-10,ZhouS-07}. We remark that the FKS and CD methods are essentially special cases of the RFKS method, in which vector $s_k$ is chosen as $s_k=(0,\ldots,0,1)^T\in\mathbb{R}^{k}$ and $s_k=y_1$, respectively. Here, $y_1$ is the eigenvector corresponding to the eigenvalue $\theta_1$ whose real part is the largest among all eigenvalues of the projected matrix $H_{k}$. In addition, in the following numerical results, the relaxed vector $s_k$ in RFKS method is chosen as that in (\ref{equation-rela-RelVecS}). Intuitively, the CD method may outperform the FKS method, as vector $V_ks_k$ in CD method contains more valuable information of the projection subspace than that in FKS method. Similarly, considering the construction of vector $s_k$ in (\ref{equation-rela-RelVecS}), the RFKS method would exhibit better numerical behavior than the FKS and CD methods for the flexibility of the polynomial filters and the efficient selection of vector $s_k$. These assertions will be verified in the subsequent numerical results.

The number of Rayleigh-Ritz processes is referred to as the count of iteration steps for each tested method. There are two parameters, i.e., the restart number $n_r$ and the order $m$ of the polynomial filter, involved in AC, CD, FKS and RFKS methods. Here, the restart number $n_r$ in AC method refers to the number of Arnoldi iterations. Different from the other three methods, the dimension of the projected Krylov subspace in AC method is fixed to $n_r$ while in CD, FKS and RFKS methods the dimensions of the projected subspaces gradually increase step by step. Thus, to be reasonable, the restart numbers in CD, FKS and RFKS methods are set to be larger than that in AC method.

For the sake of fairness, all methods are started from the vector whose entries are all set to be one, and the whole iteration process is terminated once the current relative residual norm achieves the stopping criterion $\frac{\|r^{(k)}\|}{\|r^{(0)}\|}\leq 10^{-10}$. Here, $r^{(k)}=(A-\theta^{(k)}I)x^{(k)}$ is the $k$-th residual vector with $\theta^{(k)}$ being the Rayleigh quotient associated with the approximate eigenvector $x^{(k)}$.

All numerical experiments are performed by making use of MATLAB (version R2017a) on a personal computer with $1.4$ GHz central processing unit (Intel (R) Core (TM) i5), $8.00$ GB memory, and macOS Catalina operating system (2019).

We consider the following two-dimensional partial differential system \cite{Saad-84}
\begin{align}
\label{equation-nume-exam01}
-\frac{\partial}{\partial \xu}\left(\omega(\xu,\yu)\frac{\partial \uu}{\partial \xu}\right)
-\frac{\partial}{\partial \yu}\left(\gamma(\xu,\yu)\frac{\partial \uu}{\partial \yu}\right)
+\frac{\partial}{\partial \xu}\bigg(\mu(\xu,\yu)\uu\bigg)
+\frac{\partial}{\partial \yu}\bigg(\nu(\xu,\yu)\uu\bigg)
=\lambda \uu,
\end{align}
which is defined on the domain of $[-1,1]\times[-1,1]$, and imposed with the homogeneous Dirichlet boundary condition. We discrete the problem in (\ref{equation-nume-exam01}) by the finite difference scheme on an $N\times N$ grid with the mesh size being both equal to $h=1/(N+1)$, obtaining the standard non-symmetric eigenvalue problem (\ref{equation-intr-EigeProb}). Note that the size of matrix $A$ is $n=N^2$.

The following two choices of the coefficient functions $\omega(\xu,\yu)$, $\gamma(\xu,\yu)$, $\mu(\xu,\yu)$ and $\nu(\xu,\yu)$ are considered:
\begin{align*}
\left\{
\begin{array}{l}
\text{Case (I):} \quad \omega(\xu,\yu)=-1, \gamma(\xu,\yu)=-\frac{10}{1+\xu \yu}, \mu(\xu,\yu)=1 ~\text{and}~ \nu(\xu,\yu)=\frac{1}{1+\xu\yu}; \\ [2mm]
\text{Case (II):} \quad \omega(\xu,\yu)=-e^{\xu\yu}, \gamma(\xu,\yu)=-\frac{10}{1+\xu\yu}, \mu(\xu,\yu)=\sin(1+\xu\yu) ~\text{and}~ \nu(\xu,\yu)=\frac{1}{1+\xu\yu}.
\end{array}
\right.
\end{align*}
Besides, we choose the size of matrix $A$ as $n=N^2$ with $N$ being $N_1=200$, $N_2=250$ and $N_3=300$.

In Table \ref{table-nume-exam01-tab01}, we report the numerical results for computing one eigenvalue with the largest real part and the corresponding eigenvector of the non-symmetric eigenvalue problem (\ref{equation-intr-EigeProb}). In both Cases (I) and (II), the Chebyshev order is 60, the number of Arnoldi iterations is set to be 20, and the restart number in CD, FKS and RFKS is set to be 40. We also depict the corresponding curves of the residual norm (``RES"), i.e., $\|r^{(k)}\|$, versus the number of iteration steps in Figure \ref{figure-num01-fig01}.

\begin{table}[ht]
\centering
\setlength{\abovecaptionskip}{0pt}
\setlength{\belowcaptionskip}{10pt}
\caption{\label{table-nume-exam01-tab01}
Numerical results for AC, CD, FKS and RFKS}
\begin{tabular}{|c|c|c|c|c|c|c|c|c|c|}
\hline
\multicolumn{2}{|c|}{Case} & \multicolumn{4}{|c|}{Case~(I)} & \multicolumn{4}{|c|}{Case~(II)} \\ \hline
\multicolumn{2}{|c|}{Method} & AC & CD & FKS & RFKS & AC & CD & FKS & RFKS \\ \hline
& IT &3098&2031&2735&1620&2834&1799&2557&1517 \\ \cline{2-10}
\raisebox{1.5ex}[0pt]{$N_1$}
& MV &247840&122862&169589&97980&226720&108838&158553&91774 \\ \cline{2-10}
& CPU &169.75&60.39&82.59&54.74&158.23&53.44&78.49&51.97 \\ \hline
& IT &5782&3720&3999&2258&5326&3160&4546&2259 \\ \cline{2-10}
\raisebox{1.5ex}[0pt]{$N_2$}
& MV &462560&225060&248019&136576&426080&191180&281871&136638 \\ \cline{2-10}
& CPU &456.67&180.16&197.98&118.08&419.28&154.71&228.56&118.04 \\ \hline
& IT &9608&4027&8036&4016&8822&3640&7614&3195 \\ \cline{2-10}
\raisebox{1.5ex}[0pt]{$N_3$}
& MV &768640&243614&514991&242932&705760&220220&472087&193290 \\ \cline{2-10}
& CPU &982.27&254.68&536.02&253.59&1041.54&228.69&493.55&210.76 \\ \hline
\end{tabular}
\end{table}

From Table \ref{table-nume-exam01-tab01}, we observe that the three methods CD, FKS and RFKS significantly outperform the AC method in terms of the number of iteration steps, the number of matrix-vector products and the computing time in seconds. Moreover, among the three methods, the RFKS method behaves more effective than the CD and FKS methods for consuming less iteration steps, matrix-vector products and computing time. As for the CD and FKS methods, the CD method exhibits better numerical behavior than the FKS method. In addition, it indicates that the iteration indexes, i.e., IT, MV and CPU, increase as the problem size increases.

We also use the numerical results in Table \ref{table-nume-exam01-tab02} to exhibit the numerical behavior of the four tested methods for computing one eigenvalue with the largest real part and the corresponding eigenvector of the non-symmetric eigenvalue problem (\ref{equation-intr-EigeProb}). In both Cases (I) and (II), the Chebyshev order is 60, the number of Arnoldi iterations is set to be 30, and the restart number in CD, FKS and RFKS is set to be 60. For this setting, in Figure~\ref{figure-num01-fig02} we also depict the corresponding curves of the residual norm versus the number of iteration steps for the AC, CD, FKS and RFKS methods.

\begin{table}[ht]
\centering
\setlength{\abovecaptionskip}{0pt}
\setlength{\belowcaptionskip}{10pt}
\caption{\label{table-nume-exam01-tab02}
Numerical results for AC, CD, FKS and RFKS}
\begin{tabular}{|c|c|c|c|c|c|c|c|c|c|}
\hline
\multicolumn{2}{|c|}{Case} & \multicolumn{4}{|c|}{Case~(I)} & \multicolumn{4}{|c|}{Case~(II)} \\ \hline
\multicolumn{2}{|c|}{Method} & AC & CD & FKS & RFKS & AC & CD & FKS & RFKS \\ \hline
& IT &2018&1859&1915&1503&1862&1720&1848&1479 \\ \cline{2-10}
\raisebox{1.5ex}[0pt]{$N_1$}
& MV &181620&113398&118749&91626&167580&104900&114595&90198 \\ \cline{2-10}
& CPU &155.72&59.31&61.73&51.56&147.90&55.74&60.94&53.20 \\ \hline
& IT &3756&3120&3476&2168&3452&2815&3289&2307 \\ \cline{2-10}
\raisebox{1.5ex}[0pt]{$N_2$}
& MV &338040&190320&215531&132196&310680&171710&203937&140694 \\ \cline{2-10}
& CPU &400.30&160.07&180.44&117.61&370.97&143.38&168.38&124.58 \\ \hline
& IT &6262&4500&5736&2703&5738&4496&5304&3099 \\ \cline{2-10}
\raisebox{1.5ex}[0pt]{$N_3$}
& MV &563580&274500&355651&164826&516420&274252&328867&189018 \\ \cline{2-10}
& CPU &874.59&292.33&380.82&186.62&867.08&293.32&354.70&212.36 \\ \hline
\end{tabular}
\end{table}

From Table \ref{table-nume-exam01-tab02}, we can observe the phenomena similar to those in Table \ref{table-nume-exam01-tab01}. That is, the RFKS method exhibits the best numerical behavior in terms of the number of iteration steps, the number of matrix-vector products and the computing time in seconds. Also, the proposed three methods, CD, FKS and RFKS, behave more effective than the AC method.

From Figures \ref{figure-num01-fig01} and \ref{figure-num01-fig02}, we find that among these methods the iteration counts of the RFKS method are the smallest, and those of the AC method are the largest. Far from convergence, i.e., at the start of the iteration process, the AC and RFKS methods seem to be more effective, while close to convergence the CD and RFKS methods behave more excellently than the other two methods. It also indicates that the RESs of the FKS and CD methods vibrate heavily as the residual norm decreases. In addition, closer to convergence, it looks that the RFKS method obtains a more faster convergence rate than the other methods.

\begin{figure}[ht]
\setlength{\unitlength}{1cm}
\begin{minipage}[t]{7cm}
\begin{picture}(6,5)
\includegraphics[width=7cm,height=5cm]{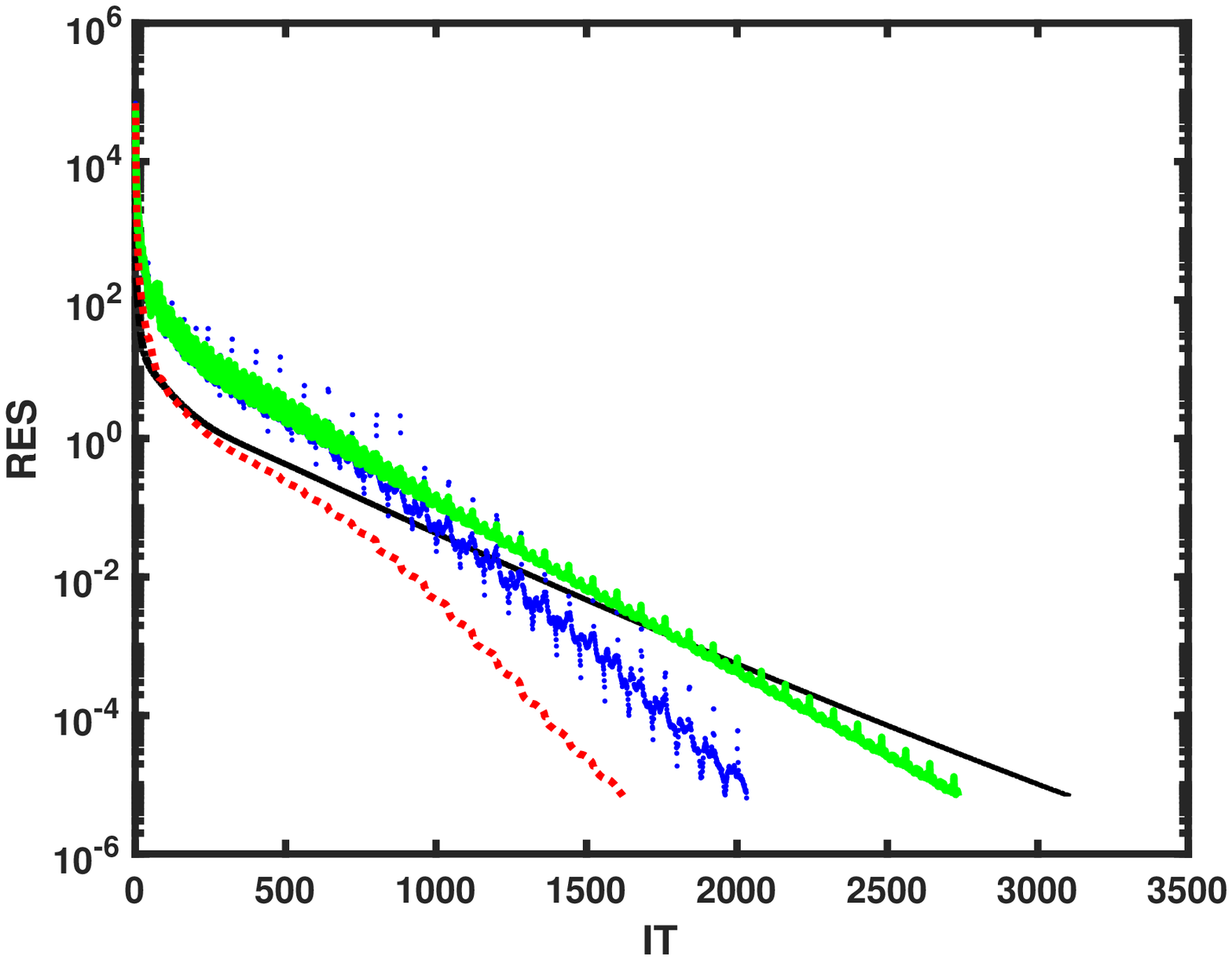}
\end{picture}\par
\end{minipage}\hfill
\begin{minipage}[t]{7cm}
\begin{picture}(6,5)
\includegraphics[width=7cm,height=5cm]{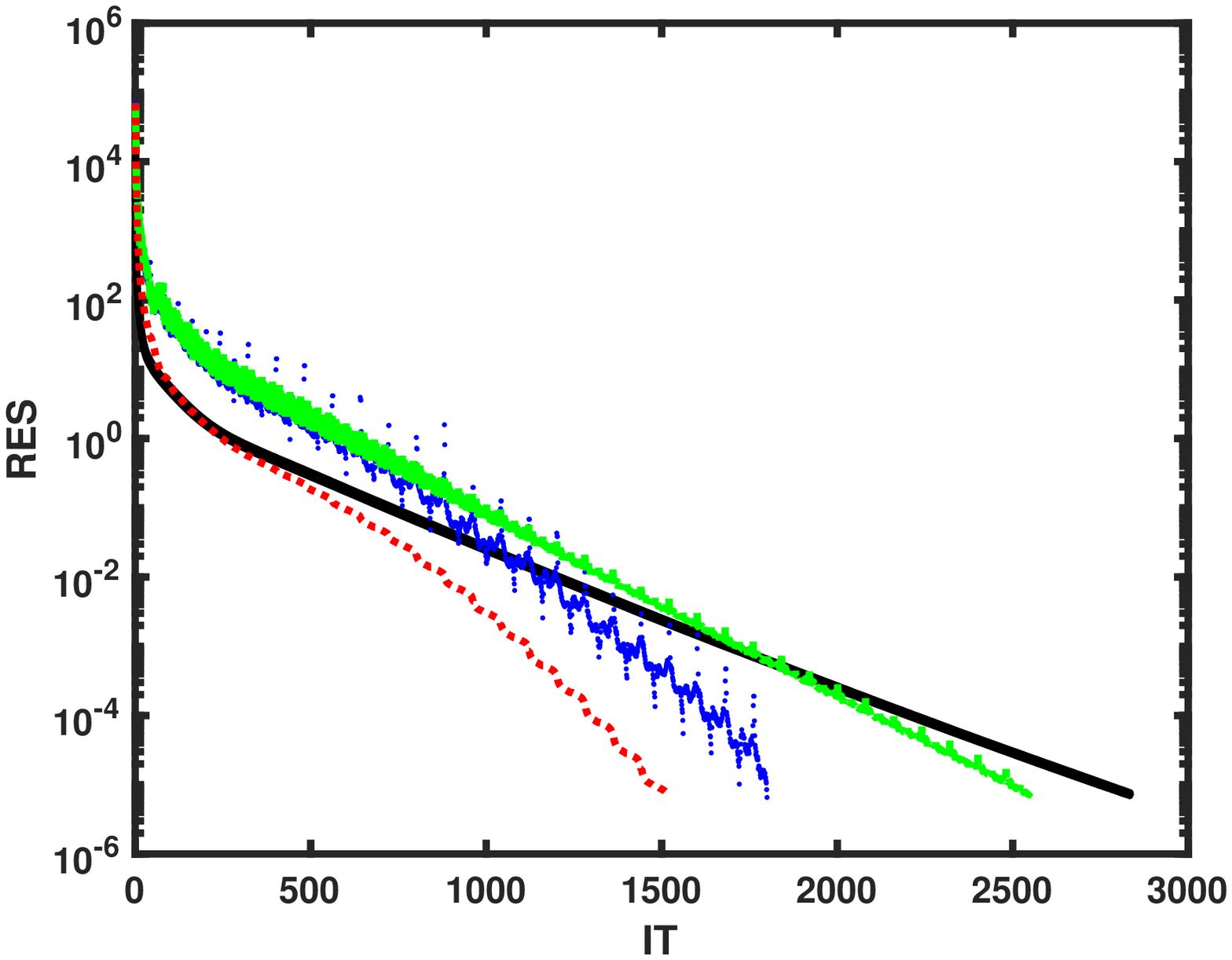}
\end{picture}\par
\end{minipage}
\caption{\label{figure-num01-fig01}
Pictures of RES versus IT when $N=N_1$ for Case (I) (left) and Case (II) (right). The black solid line ``------", the blue dotted line $``\cdot\cdot\cdot"$, the green line $``-\cdot-\cdot-\cdot"$ and the red dashed line $``---"$ represent the AC, CD, FKS and RFKS methods, respectively.
}
\end{figure}

\begin{figure}[ht]
\setlength{\unitlength}{1cm}
\begin{minipage}[t]{7cm}
\begin{picture}(6,5)
\includegraphics[width=7cm,height=5cm]{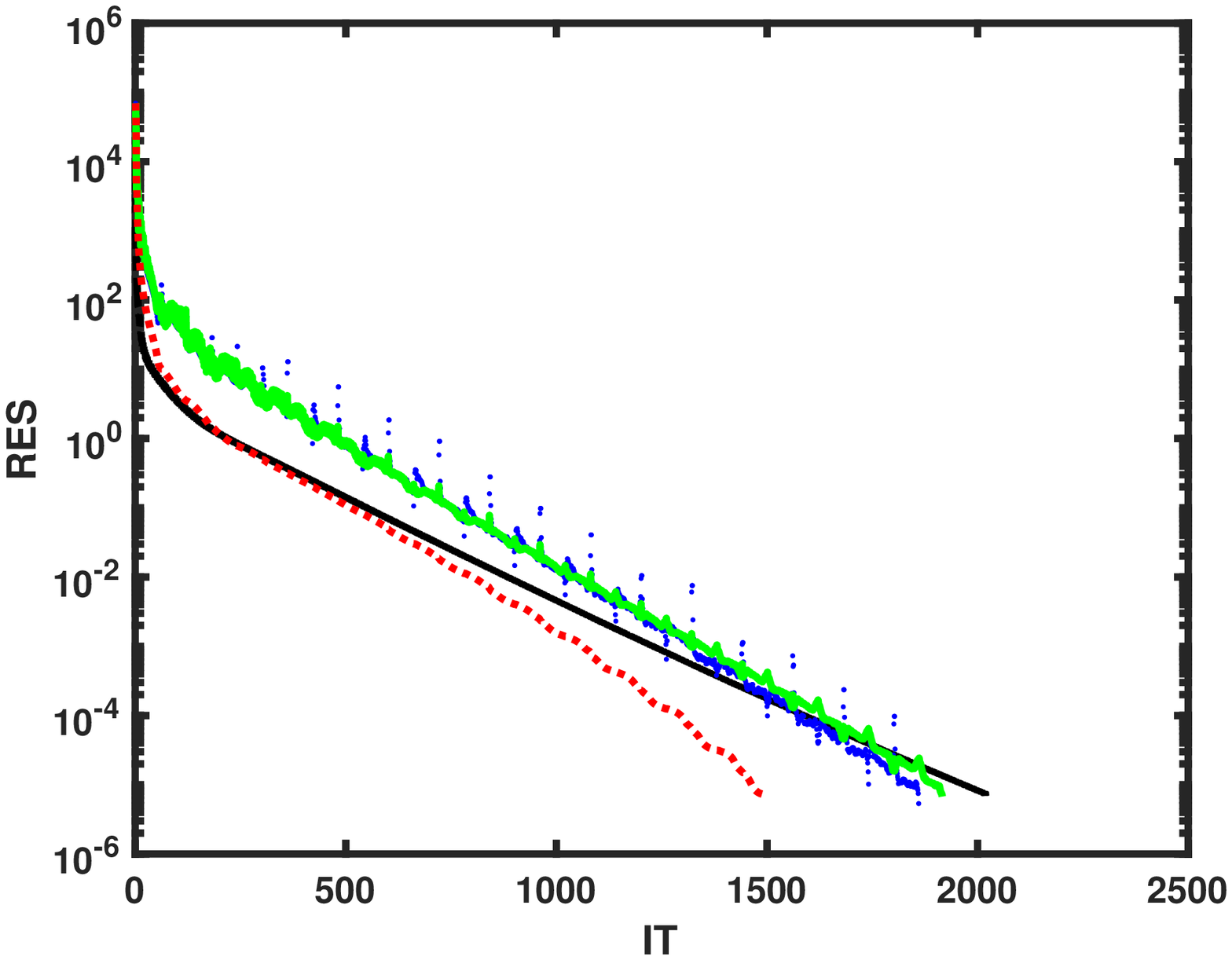}
\end{picture}\par
\end{minipage}\hfill
\begin{minipage}[t]{7cm}
\begin{picture}(6,5)
\includegraphics[width=7cm,height=5cm]{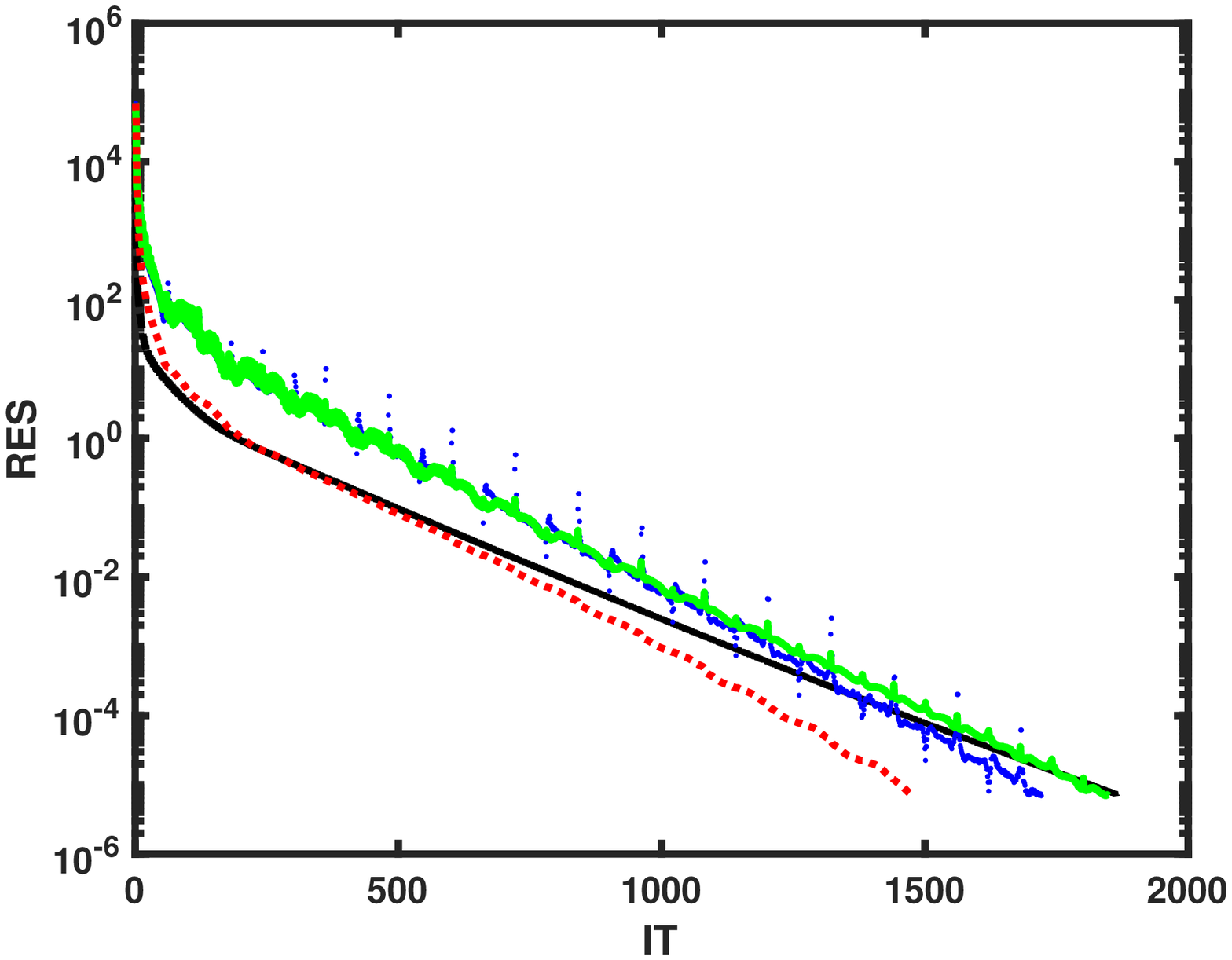}
\end{picture}\par
\end{minipage}
\caption{\label{figure-num01-fig02}
Pictures of RES versus IT when $N=N_1$ for Case (I) (left) and Case (II) (right). The black solid line ``------", the blue dotted line $``\cdot\cdot\cdot"$, the green line $``-\cdot-\cdot-\cdot"$ and the red dashed line $``---"$ represent the AC, CD, FKS and RFKS methods, respectively.
}
\end{figure}

We use Figure \ref{figure-num01-fig03} to depict the curves of the number of matrix-vector products versus the restart number. We can see from this figure that for a same polynomial order and a same restart number, the RFKS method outperforms the CD and FKS methods in terms of the number of matrix-vector products. Moreover, the number of matrix-vector products as well as the number of iteration steps decrease obviously as the restart number increases. However, as long as the restart number is large enough, the three methods may behave comparatively.

In addition, we use Figure \ref{figure-num01-fig04} to depict the curves of the number of matrix-vector products versus the polynomial order. From this figure, we can also observe that for a same polynomial order and a same restart number, the RFKS method outperforms the CD and FKS methods in terms of the number of matrix-vector products. However, it looks that in these numerical experiments increasing the polynomial order blindly may be meaningless. Unfortunately, we are unable to provide the optimal polynomial order no matter in theory or in practice.

\begin{figure}[ht]
\setlength{\unitlength}{1cm}
\begin{minipage}[t]{7cm}
\begin{picture}(6,5)
\includegraphics[width=7cm,height=5cm]{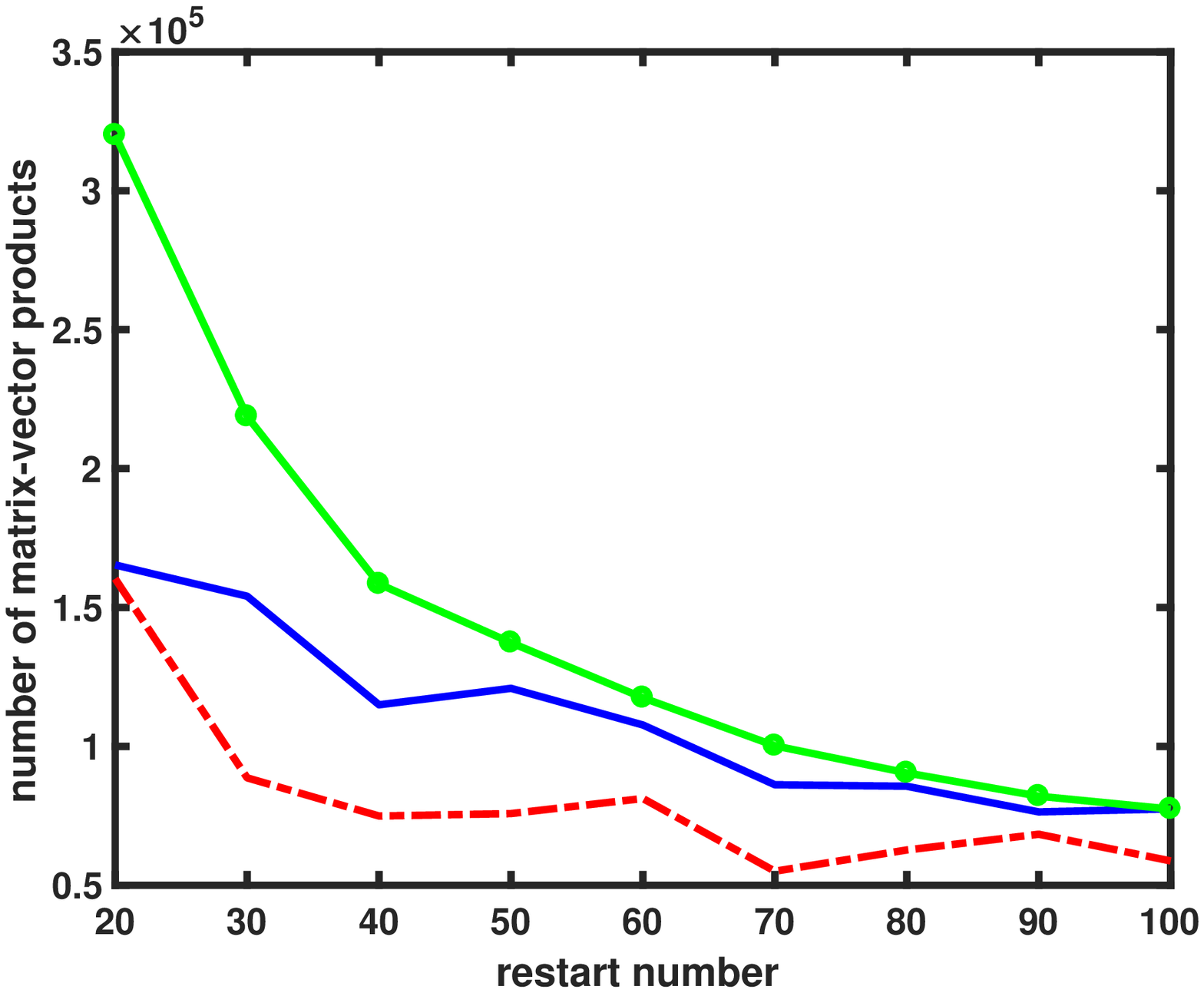}
\end{picture}\par
\end{minipage}\hfill
\begin{minipage}[t]{7cm}
\begin{picture}(6,5)
\includegraphics[width=7cm,height=5cm]{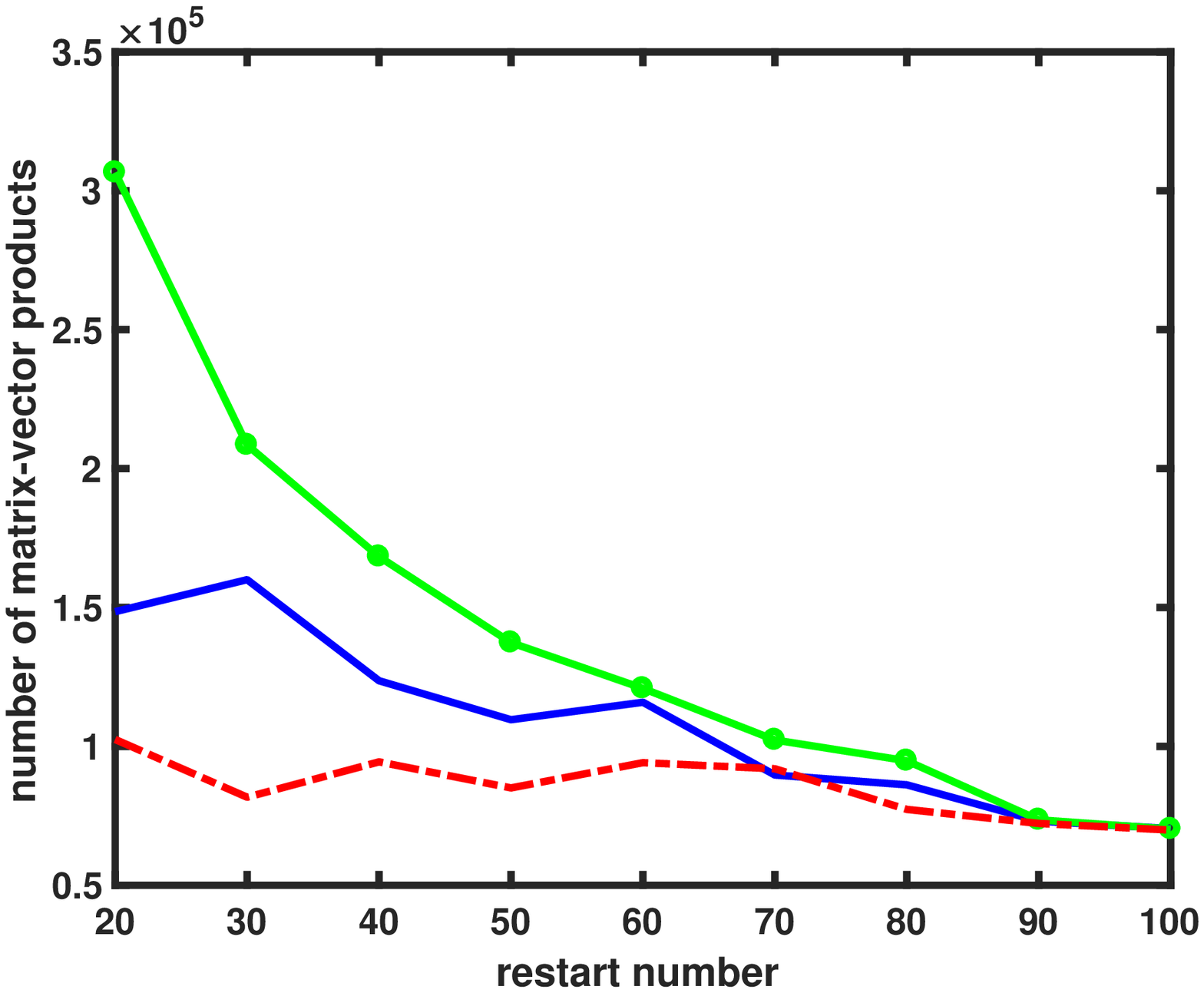}
\end{picture}\par
\end{minipage}
\caption{\label{figure-num01-fig03}
Pictures of the number of matrix-vector products versus the restart number when $N=N_1$ for Case (I) with the polynomial order being 30 (left) and being 50 (right). The green line $``-\circ-\circ-\circ"$, the blue solid line ``------" and the red line $``-\cdot-\cdot-\cdot"$ represent the FKS, CD and RFKS methods, respectively.
}
\end{figure}

\begin{figure}[ht]
\setlength{\unitlength}{1cm}
\begin{minipage}[t]{7cm}
\begin{picture}(6,5)
\includegraphics[width=7cm,height=5cm]{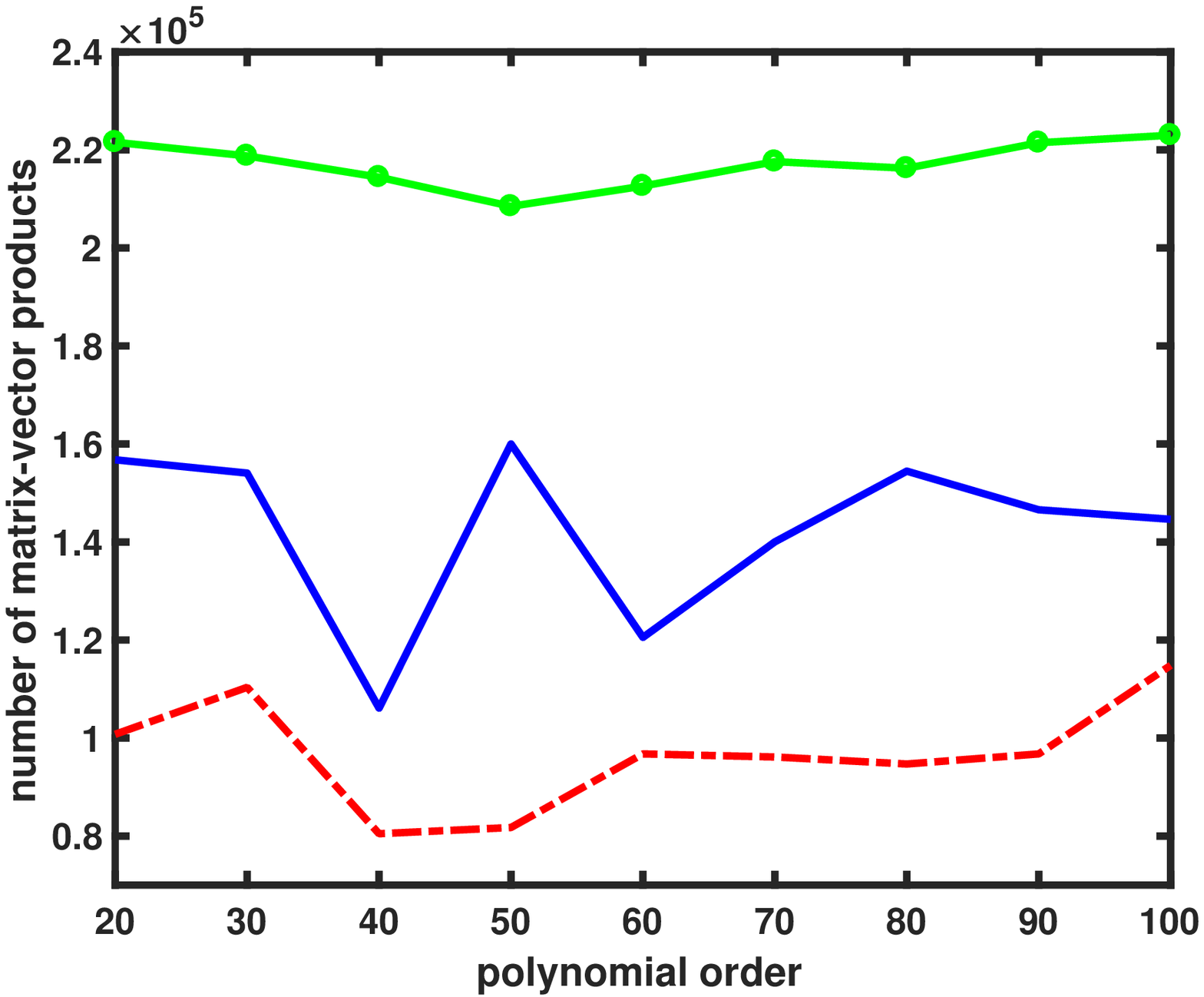}
\end{picture}\par
\end{minipage}\hfill
\begin{minipage}[t]{7cm}
\begin{picture}(6,5)
\includegraphics[width=7cm,height=5cm]{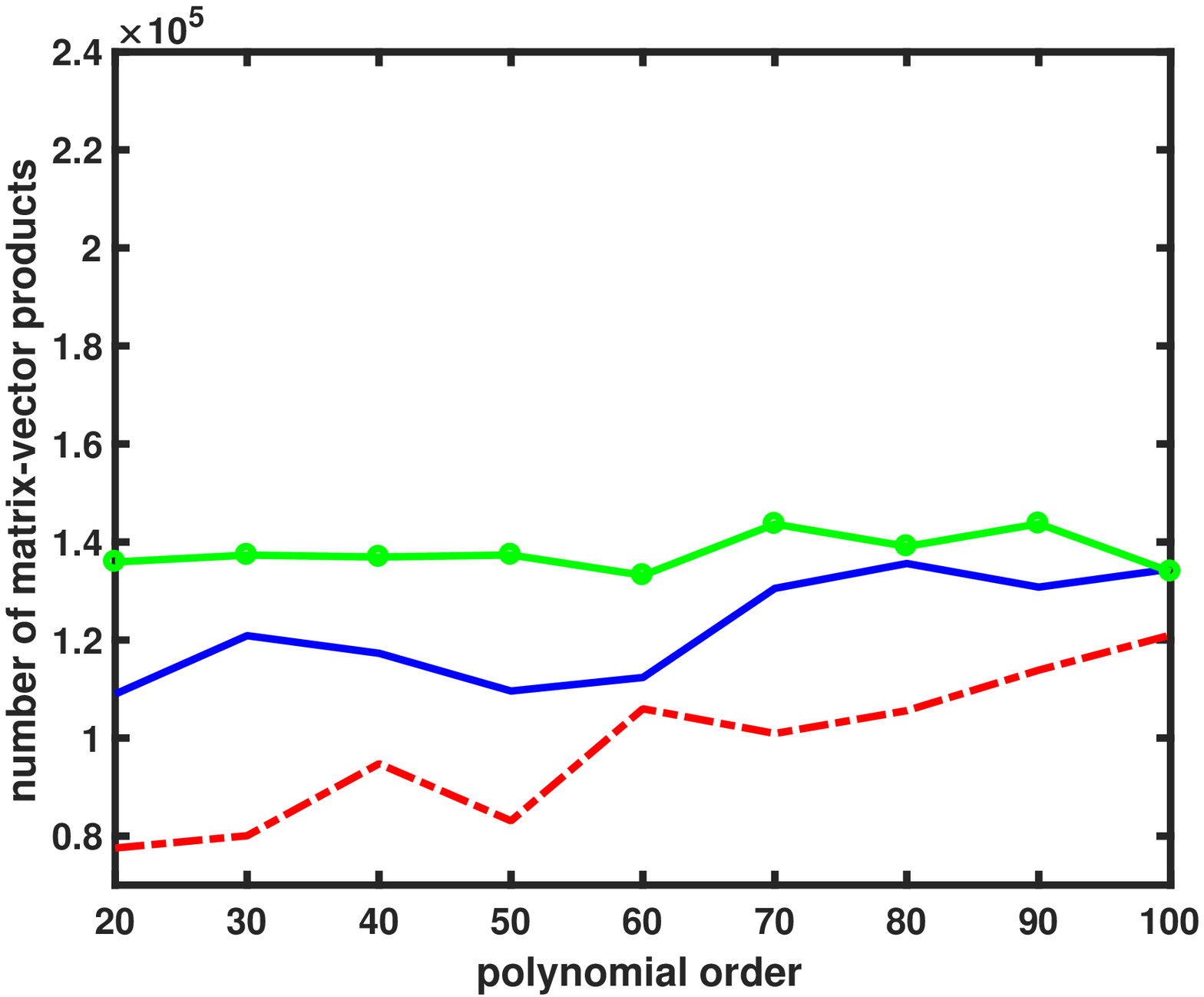}
\end{picture}\par
\end{minipage}
\caption{\label{figure-num01-fig04}
Pictures of the number of matrix-vector products versus the polynomial order when $N=N_1$ for Case (I) with the restart number being 30 (left) and being 50 (right). The green line $``-\circ-\circ-\circ"$, the blue solid line ``------" and the red line $``-\cdot-\cdot-\cdot"$ represent the FKS, CD and RFKS methods, respectively.
}
\end{figure}

\section{Concluding Remarks}
In this paper, we propose the relaxed filtered Krylov subspace method as well as the filtered Krylov subspace method and the Chebyshev-Davidson method for computing the eigenvalues with the largest real parts and the corresponding eigenvectors of non-symmetric matrices. We should remark that these methods are also suitable for solving non-Hermitian eigenvalue problems.

Admittedly, the two parameters of the restart number and the polynomial order, crucially affect the convergence property of the relaxed filtered Krylov subspace method. Large restart numbers and polynomial orders may accelerate the convergence of these methods, while, on the other hand, they would also result in huge computational costs. Consequently, there is a trade-off in the choice of these two parameters. However, their optimal and practical choices are extremely challenging problems in both theory and practice.

We should remark that the pair of eigenvectors associated with a conjugate pair of single eigenvalues are conjugate to each other. Therefore, in actual computations when encountered with complex eigenvalues, we can restrict the desired eigenvalue with its imaginary part being positive or negative so that the computed eigenvalue is single. However, when encountered with multiple and clustered eigenvalues, it prefers to work with a block of vectors instead of a single vector for its effectiveness and robustness. Block method for several eigenvalues is generally more complicated and difficult than a single vector form. Thus, future research in this issue is worth pursuing.

\clearpage

\section{Appendix}
\label{sect-append}

\begin{proof}[Proof of Lemma \ref{lemma-cheb-inequalities}]
First, we illuminate the validity of item ($i$). Using the facts $|\sqrt{\xi}|=\sqrt{|\xi|}$ and $|\xi^2|=|\xi|^2$ for any $\xi\in\mathbb{C}$, we have
\begin{align*}
|z+\sqrt{z^2-1}|\leq |z|+|\sqrt{z^2-1}|=|z|+\sqrt{|z^2-1|}\leq |z|+\sqrt{|z^2|+1}= |z|+\sqrt{|z|^2+1}.
\end{align*}
This implies that the third inequality of item ($i$) holds true.

Obviously, when $z=0$ the first and second inequalities of item ($i$) are valid. Thus, we now suppose that $z \neq 0$.

Since when $\Re(z)>0$, or when $\Re(z)=0$ and $\Im(z)>0$, it holds that $z=\sqrt{z^2}$, thus, by dividing both sides of the first and second inequalities of item ($i$) by $|z|$, we can equivalently obtain
\begin{align*}
\left|1-\sqrt{1-\frac{1}{z^2}}\right|
\leq \left|1+\sqrt{1-\frac{1}{|z|^2}}\right|
\leq \left|1+\sqrt{1-\frac{1}{z^2}}\right|.
\end{align*}
This means that, for any nonzero $\tilde{z}\in\mathbb{C}$, we just need to illuminate the validity of
\begin{align}
\label{equation-cheb-LemIneq01}
\left|1-\sqrt{1-\tilde{z}^2}\right|
\leq \left|1+\sqrt{1-|\tilde{z}|^2}\right|
\leq \left|1+\sqrt{1-\tilde{z}^2}\right|.
\end{align}

Denote by
\begin{align*}
\tilde{z}=|\tilde{z}|
(\cos\theta+i\sin\theta)
\quad\mbox{and}\quad
1-\tilde{z}^2=|1-\tilde{z}^2|
(\cos\eta+i\sin\eta),
\end{align*}
where
\begin{align*}
\cos\eta=\frac{1-|\tilde{z}|^2\cos2\theta}{|1-\tilde{z}^2|}
\quad\mbox{and}\quad
\sin\eta=-\frac{|\tilde{z}|^2\sin2\theta}{|1-\tilde{z}^2|}.
\end{align*}
Then, it holds that
\begin{align*}
\sqrt{1-\tilde{z}^2}
=\sqrt{|1-\tilde{z}^2|}
\, (\cos\tilde{\eta}+i\sin\tilde{\eta}),
\end{align*}
\begin{align}
\label{equation-cheb-LemIneq02}
\left|1-\sqrt{1-\tilde{z}^2}\right|
&=\left|1-\sqrt{|1-\tilde{z}^2|}
\, (\cos\tilde{\eta}+i\sin\tilde{\eta})\right| \notag \\
&=\sqrt{1+|1-\tilde{z}^2|-2\sqrt{|1-\tilde{z}^2|} \, \cos\tilde{\eta}},
\end{align}
and
\begin{align}
\label{equation-cheb-LemIneq03}
\left|1+\sqrt{1-\tilde{z}^2}\right|
&=\left|1+\sqrt{|1-\tilde{z}^2|}
\, (\cos\tilde{\eta}+i\sin\tilde{\eta})\right| \notag \\
&=\sqrt{1+|1-\tilde{z}^2|+2\sqrt{|1-\tilde{z}^2|} \, \cos\tilde{\eta}},
\end{align}
where $\cos\tilde{\eta}\geq 0$, $\cos2\tilde{\eta}=\cos\eta$, and $\sin2\tilde{\eta}=\sin\eta$.

If $|\tilde{z}| \geq 1$, then
\begin{align}
\label{equation-cheb-LemIneq04}
\left|1+\sqrt{1-|\tilde{z}|^2}\right|
=\left|1+\sqrt{|\tilde{z}|^2-1} \, i\right|
=\sqrt{1+|\tilde{z}|^2-1}
=|\tilde{z}|.
\end{align}
Therefore, to obtain the inequalities in (\ref{equation-cheb-LemIneq01}), according to the equalities (\ref{equation-cheb-LemIneq02}), (\ref{equation-cheb-LemIneq03}) and (\ref{equation-cheb-LemIneq04}), we just need to prove
\begin{align*}
\sqrt{1+|1-\tilde{z}^2|-2\sqrt{|1-\tilde{z}^2|} \, \cos\tilde{\eta}}
\leq |\tilde{z}|
\leq \sqrt{1+|1-\tilde{z}^2|+2\sqrt{|1-\tilde{z}^2|} \, \cos\tilde{\eta}},
\end{align*}
that is,
\begin{align}
\label{equation-cheb-LemIneq05}
1+|1-\tilde{z}^2|-2\sqrt{|1-\tilde{z}^2|} \, \cos\tilde{\eta}
\leq |\tilde{z}|^2
\leq 1+|1-\tilde{z}^2|+2\sqrt{|1-\tilde{z}^2|} \, \cos\tilde{\eta}.
\end{align}

Since
\begin{align*}
\cos^2\tilde{\eta}
=\frac{1+\cos2\tilde{\eta}}{2}
=\frac{1+\cos\eta}{2}
=\frac{1}{2}+\frac{1-|\tilde{z}|^2\cos2\theta}{2|1-\tilde{z}^2|},
\end{align*}
we have
\begin{align}
\label{equation-cheb-LemIneq06}
|1-\tilde{z}^2|\cos^2\tilde{\eta}
=\frac{|1-\tilde{z}^2|}{2}
+\frac{1-|\tilde{z}|^2\cos2\theta}{2}.
\end{align}
From
\begin{align}
\label{equation-cheb-LemIneq07}
|1-\tilde{z}^2| \leq 1+|\tilde{z}^2|=1+|\tilde{z}|^2
\quad\text{and}\quad
|1-\tilde{z}^2|
\geq |\tilde{z}^2|-1
=|\tilde{z}|^2-1,
\end{align}
we obtain
\begin{align}
\label{equation-cheb-LemIneq08}
\sqrt{1+|1-\tilde{z}^2|-|\tilde{z}|^2}\leq\sqrt{2}.
\end{align}
Thus, it holds that
\begin{align*}
1+|1-\tilde{z}^2|-|\tilde{z}|^2
& \leq \sqrt{2(1+|1-\tilde{z}^2|-|\tilde{z}|^2)} \\
& \leq \sqrt{2(1+|1-\tilde{z}^2|-|\tilde{z}|^2\cos2\theta)}.
\end{align*}
Then, from the equality in (\ref{equation-cheb-LemIneq06}) and $\cos\tilde{\eta} \geq 0$, we can obtain
\begin{align*}
1+|1-\tilde{z}^2|-|\tilde{z}|^2
\leq 2\sqrt{|1-\tilde{z}^2|} \, \cos\tilde{\eta},
\end{align*}
which is just the first inequality of (\ref{equation-cheb-LemIneq05}).

From $\cos\tilde{\eta} \geq 0$ and the second inequality in (\ref{equation-cheb-LemIneq07}), we have
\begin{align*}
1+|1-\tilde{z}^2|+2\sqrt{|1-\tilde{z}^2|} \, \cos\tilde{\eta}
\geq 1+|1-\tilde{z}^2|
\geq 1+|\tilde{z}|^2-1
=|\tilde{z}|^2,
\end{align*}
which deduces the validity of the second inequality of (\ref{equation-cheb-LemIneq05}).

If $|\tilde{z}|<1$, it holds that $\left|1+\sqrt{1-|\tilde{z}|^2}\right|
=1+\sqrt{1-|\tilde{z}|^2}$. Therefore, from the equalities (\ref{equation-cheb-LemIneq02}) and (\ref{equation-cheb-LemIneq03}), we know that to obtain the inequalities in (\ref{equation-cheb-LemIneq01}), we just need to prove
\begin{align*}
\sqrt{1+|1-\tilde{z}^2|-2\sqrt{|1-\tilde{z}^2|} \, \cos\tilde{\eta}}
&\leq 1+\sqrt{1-|\tilde{z}|^2} \\
&\leq \sqrt{1+|1-\tilde{z}^2|+2\sqrt{|1-\tilde{z}^2|} \, \cos\tilde{\eta}},
\end{align*}
that is,
\begin{align}
\label{equation-cheb-LemIneq09}
|1-\tilde{z}^2|-2\sqrt{|1-\tilde{z}^2|} \, \cos\tilde{\eta}
&\leq 1-|\tilde{z}|^2+2\sqrt{1-|\tilde{z}|^2} \notag\\
&\leq |1-\tilde{z}^2|+2\sqrt{|1-\tilde{z}^2|} \, \cos\tilde{\eta}.
\end{align}

Using the condition $|\tilde{z}|<1$, and the equalities (\ref{equation-cheb-LemIneq06}) and (\ref{equation-cheb-LemIneq08}), we have
\begin{align*}
|1-\tilde{z}^2|-1+|\tilde{z}|^2
& \leq |1-\tilde{z}^2|+1-|\tilde{z}|^2 \\
& \leq \sqrt{2(|1-\tilde{z}^2|+1-|\tilde{z}|^2)} \\
& \leq \sqrt{2(|1-\tilde{z}^2|+1-|\tilde{z}|^2\cos2\theta)}
+2\sqrt{1-|\tilde{z}|^2} \\
& = 2\sqrt{|1-\tilde{z}^2|} \, \cos\tilde{\eta}
+2\sqrt{1-|\tilde{z}|^2},
\end{align*}
which is just the first inequality of (\ref{equation-cheb-LemIneq09}).

Since
\begin{align}
\label{equation-cheb-LemIneq10}
|1-\tilde{z}^2|
\geq 1-|\tilde{z}^2|
= 1-|\tilde{z}|^2,
\end{align}
from equality (\ref{equation-cheb-LemIneq06}), we have
\begin{align*}
|1-\tilde{z}^2|\cos^2\tilde{\eta}
&\geq \frac{|1-\tilde{z}^2|}{2}
+\frac{1-|\tilde{z}|^2}{2}\geq 1-|\tilde{z}|^2.
\end{align*}
Then, it follows from $\cos\tilde{\eta}\geq 0$ and $|\tilde{z}|<1$ that
\begin{align*}
\sqrt{|1-\tilde{z}^2|} \, \cos\tilde{\eta}
\geq \sqrt{1-|\tilde{z}|^2}.
\end{align*}
Combining this inequality with inequality (\ref{equation-cheb-LemIneq10}), we can easily obtain the second inequality of (\ref{equation-cheb-LemIneq09}). As a result, we have now already proved item ($i$).

With the validity of item ($i$), we can easily demonstrate the validity of item ($ii$). As if the real and imaginary parts of $z$ satisfy $\Re(z)<0$, or satisfy $\Re(z)=0$ and $\Im(z) \leq 0$, by replacing $z$ by $-z$ in the inequalities in item ($i$) we can obtain inequalities of item ($ii$). Specifically, it holds that
\begin{align*}
\left|-z-\sqrt{(-z)^2-1}\right|
 \leq \left||-z|+\sqrt{|-z|^2-1}\right|
 \leq \left|-z+\sqrt{(-z)^2-1}\right|
\leq |-z|+\sqrt{|-z|^2+1},
\end{align*}
which are just the inequalities in item ($ii$).
\end{proof}

%%%%%%%%%%%%%%%%% References %%%%%%%%%%%%
%\clearpage

%%%%%%%%%%%%%%%%%%%%%%%%%%%%%%%%%%%%%%%
%%%%%%%%%%%%%%%%%%%%%%%%%%%%%%%%%%%%%%%


\begin{thebibliography}{99}

\bibitem{Anderson-10}
C.R. Anderson,
\newblock
A Rayleigh-Chebyshev procedure for finding the smallest eigenvalues and associated eigenvectors of large sparse Hermitian matrices,
\newblock{\em J. Comput. Phys.},
229(2010), 7477-7487.

\bibitem{BaiM-NA19}
Z.-Z. Bai and C.-Q. Miao,
\newblock
Computing eigenpairs of Hermitian matrices in perfect Krylov subspaces,
\newblock{\em Numer. Algorithms},
82(2019), 1251-1277.

\bibitem{BaiW-20}
Z.-Z. Bai, W.-T. Wu and G.V. Muratova,
\newblock
The power method and beyond,
\newblock{\em Appl. Numer. Math.},
(2020), DOI: https://doi.org/10.1016/j.apnum.2020.03.021.

\bibitem{BekasKS-08}
C. Bekas, E. Kokiopoulou and Y. Saad,
\newblock
Computation of large invariant subspaces using polynomial filtered Lanczos iterations with applications in density functional theory,
\newblock{\em SIAM J. Matrix Anal. Appl.},
30(2008), 397-418.

\bibitem{CrouzeixPS-94}
M. Crouzeix, B. Philippe and M. Sadkane,
\newblock
The Davidson method,
\newblock{\em SIAM J. Sci. Comput.},
15(1994), 62-76.

\bibitem{Davidson-75}
E.R. Davidson,
\newblock
The iterative calculation of a few of the lowest eigenvalues and corresponding eigenvectors of large real-symmetric matrices,
\newblock{\em J. Comput. Phys.},
17(1975), 87-94.

\bibitem{FangS-12}
H.R. Fang and Y. Saad,
\newblock
A filtered Lanczos procedure for extreme and interior eigenvalue problems,
\newblock{\em SIAM J. Sci. Comput.},
34(2012), A2220-A2246.

\bibitem{Jia-00}
Z.-X. Jia,
\newblock
A refined subspace iteration algorithm for large sparse eigenproblems,
\newblock{\em Appl. Numer. Math.},
32(2000), 35-52.

\bibitem{Knyazev-01}
A.V. Knyazev,
\newblock
Toward the optimal preconditioned eigensolver: Locally optimal block preconditioned conjugate gradient method,
\newblock{\em SIAM J. Sci. Comput.},
23(2001), 517-541.

\bibitem{Miao-EAJAM17}
C.-Q. Miao,
\newblock
A filtered-Davidson method for large symmetric eigenvalue problems,
\newblock{\em East Asian J. Appl. Math.},
7(2017), 21-34.

\bibitem{Miao-JCAM18}
C.-Q. Miao,
\newblock
Computing eigenpairs in augmented Krylov subspace produced by Jacobi-Davidson correction equation,
\newblock{\em J. Comput. Appl. Math.},
343(2018), 363-372.

\bibitem{Miao-NA19}
C.-Q. Miao,
\newblock
Filtered Krylov-like sequence method for symmetric eigenvalue problems,
\newblock{\em Numer. Algorithms},
82(2019), 791-807.

\bibitem{Morgan-92}
R.B. Morgan,
\newblock
Generalizations of Davidson's method for computing eigenvalues of large nonsymmetric matrices,
\newblock{\em J. Comput. Phys.},
101(1992), 287-291.

\bibitem{Ovtchinnikov-SIAM06}
E.E. Ovtchinnikov,
\newblock
Sharp convergence estimates for the preconditioned steepest descent method for Hermitian eigenvalue problems,
\newblock{\em SIAM J. Numer. Anal.},
43(2006), 2668-2689.

\bibitem{Parlett-98}
B.N. Parlett,
\newblock
The Symmetric Eigenvalue Problem,
\newblock{\em SIAM},
Philadelphia, PA, 1998.

\bibitem{Ruhe-98}
A. Ruhe,
\newblock
Rational Krylov: A practical algorithm for large sparse nonsymmetric matrix pencils,
\newblock{\em SIAM J. Sci. Comput.},
19(1998), 1535-1551.

\bibitem{Saad-84}
Y. Saad,
\newblock
Chebyshev acceleration techniques for solving nonsymmetric eigenvalue problems,
\newblock{\em Math. Comp.},
42(1984), 567-588.

\bibitem{Saad-11}
Y. Saad,
\newblock
Numerical Methods for Large Eigenvalue Problems, Second Edition,
\newblock{\em SIAM},
Philadelphia, PA, 2011.

\bibitem{Sadkane-93}
M. Sadkane,
\newblock
A block Arnoldi-Chebyshev method for computing the leading eigenpairs of large sparse unsymmetric matrices,
\newblock{\em Numer. Math.},
64(1993), 181-193.

\bibitem{SleijpenV-96}
G.L.G. Sleijpen and H.A. Van der Vorst,
\newblock
A Jacobi-Davidson iteration method for linear eigenvalue problems,
\newblock{\em SIAM J. Matrix Anal. Appl.},
17(1996), 401-425.

\bibitem{Sorensen-92}
D.C. Sorensen,
\newblock
Implicit application of polynomial filters in a $k$-step Arnoldi method,
\newblock{\em SIAM J. Matrix Anal. Appl.},
13(1992), 357-385.

\bibitem{VecharynskiYP-15}
E. Vecharynski, C. Yang and J.E. Pask,
\newblock
A projected preconditioned conjugate gradient algorithm for computing many extreme eigenpairs of a Hermitian matrix,
\newblock{\em J. Comput. Phys.},
290(2015), 73-89.

\bibitem{XiS-16}
Y.-Z. Xi and Y. Saad,
\newblock
Computing partial spectra with least-squares rational filters,
\newblock{\em SIAM J. Sci. Comput.},
38(2016), A3020-A3045.

\bibitem{Zhou-10}
Y.-K. Zhou,
\newblock
A block Chebyshev-Davidson method with inner-outer restart for large eigenvalue problems,
\newblock{\em J. Comput. Phys.},
229(2010), 9188-9200.

\bibitem{ZhouS-07}
Y.-K. Zhou and Y. Saad,
\newblock
A Chebyshev-Davidson algorithm for large symmetric eigenproblems,
\newblock{\em SIAM J. Matrix Anal. Appl.},
29(2007), 954-971.

\end{thebibliography}
\end{document}